\pdfoutput=1
\documentclass{article}
\usepackage[a4paper,marginratio=1:1,width=150mm,top=25mm,bottom=25mm]{geometry}
\usepackage[T1]{fontenc}
\usepackage[utf8]{inputenc}

\usepackage{amsmath}
\usepackage{amssymb}
\usepackage{amsthm}
\usepackage{amsfonts}

\usepackage{tikz}
\usetikzlibrary{calc}
\usepackage{tikz-cd}
\usepackage{todonotes}

\usepackage{quiver}

\usepackage{ytableau}

\usepackage{spectralsequences}

\usepackage{graphicx}
\usepackage{xcolor}
\usepackage{hyperref}
\usepackage{cleveref}
\usepackage[citestyle=alphabetic,bibstyle=alphabetic,maxbibnames=9]{biblatex}
\usepackage{parskip}
\usepackage{tcolorbox}
\tcbuselibrary{most}

\definecolor{RUBBlue}{RGB}{128,0,50}
\definecolor{RUBGrey}{RGB}{231, 231, 231}
\definecolor{RUBLightGrey}{RGB}{252, 252, 252}
\definecolor{RUBGreen}{RGB}{0, 49, 83}

\theoremstyle{plain}
\newtheorem{theorem}{Theorem}[section]
\crefname{theorem}{{Theorem}}{{Theorems}}
\Crefname{theorem}{{Theorem}}{{Theorems}}
\newtheorem{theoremalph}{Theorem}

\crefname{theoremalph}{{Theorem}}{{Theorems}}
\Crefname{theoremalph}{{Theorem}}{{Theorems}}
\newtheorem{proposition}[theorem]{Proposition}
\crefname{proposition}{{Proposition}}{{Propositions}}
\Crefname{proposition}{{Proposition}}{{Propositions}}
\newtheorem{lemma}[theorem]{Lemma}
\crefname{lemma}{{Lemma}}{{Lemmas}}
\Crefname{lemma}{{Lemma}}{{Lemmas}}
\newtheorem{corollary}[theorem]{Corollary}
\crefname{corollary}{{Corollary}}{{Corollaries}}
\Crefname{corollary}{{Corollary}}{{Corollaries}}

\theoremstyle{definition}
\newtheorem{definition}[theorem]{Definition}
\crefname{definition}{{Definition}}{{Definitions}}
\Crefname{definition}{{Definition}}{{Definitions}}
\newtheorem{construction}[theorem]{Construction}
\crefname{construction}{{Construction}}{{Constructions}}
\Crefname{construction}{{Construction}}{{Constructions}}
\newtheorem{notation}[theorem]{Notation}
\crefname{notation}{{Notation}}{{Notations}}
\Crefname{notation}{{Notation}}{{Notations}}

\theoremstyle{remark}

\crefname{example}{{Example}}{{Examples}}
\Crefname{example}{{Example}}{{Examples}}
\newtheorem{remark}[theorem]{Remark}
\crefname{remark}{{Remark}}{{Remarks}}
\Crefname{remark}{{Remark}}{{Remarks}}
\Crefname{part}{{\textsection}\!}{{\textsection}\!}
\Crefname{chapter}{{\textsection}\!}{{\textsection}\!}
\Crefname{section}{{\textsection}\!}{{\textsection}\!}
\Crefname{subsection}{{\textsection}\!}{{\textsection}\!}
\Crefname{appendix}{{\textsection}\!}{{\textsection}\!}

\tcolorboxenvironment{theorem}{colback=white, colframe = RUBBlue}
\tcolorboxenvironment{proposition}{colback=white, colframe = RUBBlue}
\tcolorboxenvironment{lemma}{colback=white, colframe = RUBBlue}
\tcolorboxenvironment{corollary}{colback=white, colframe = RUBBlue}
\tcolorboxenvironment{definition}{colframe= RUBGreen, colback=white}
\tcolorboxenvironment{construction}{colframe= RUBGreen, colback=white,breakable}
\tcolorboxenvironment{notation}{colframe= RUBGreen, colback=white}
\tcolorboxenvironment{convention}{colframe= RUBGreen, colback=white}
\tcolorboxenvironment{proof}{blanker,breakable,left=1mm,before skip=10pt,after skip=10pt
}
\tcolorboxenvironment{remark}{blanker,breakable,
before skip=10pt,after skip=10pt}

\DeclareMathOperator*{\colim}{colim}

\DeclareMathOperator{\id}{id}

\DeclareMathOperator{\hm}{Hom}

\renewcommand{\sb}[1]{\left[#1 \right]}
\newcommand{\nb}[1]{\left(#1 \right)}
\newcommand{\pb}[1]{{[\![ #1 ]\!]}}

\newcommand{\ff}{\mathbb{F}}
\newcommand{\ep}{\epsilon}

\newcommand{\cc}{\mathbb{C}}
\newcommand{\zz}{\mathbb{Z}}

\newcommand{\mm}{\mathfrak{m}}
\newcommand{\einfty}{\mathbb{E}_\infty}
\renewcommand{\gg}{\mathbb{G}}
\renewcommand{\ss}{\mathbb{S}}

\begin{document}
\title{\(K(2)\)-local splittings of finite Galois extensions of \(MU\langle6\rangle\) and \(MString\)}
\author{Leonard Tokic}
\maketitle
\begin{abstract}
  Using a Milnor-Moore argument we show that, \(K(2)\)-locally at the prime \(2\), the spectra \(MU\langle 6\rangle\) and \(MString\) split as direct sums of Morava \(E\)-theories after tensoring with a finite Galois extension of the sphere called \(E^{hF_{3/2}}\).
  In the case of \(MString\) we are able to refine this splitting in several ways: we show that the projection maps are determined by spin characteristic classes, that the Ando-Hopkins-Rezk orientation admits a unital section after tensoring with \(E^{hF_{3/2}}\), and that the splitting can be improved to one of \(E^{hH}\otimes MString\) into a direct sum of shifts of \(TMF_0(3)\) where \(H\) is an open subgroup of the Morava stabilizer group of index 4.
\end{abstract}

\setcounter{tocdepth}{3}
%\tableofcontents
%\newpage

\addcontentsline{toc}{section}{Introduction}
\section*{Introduction}
In \cite{AHR_orientation} Ando, Hopkins, and Rezk construct an \(\einfty\)-ring map
\[\tau_W:MString\longrightarrow tmf\]
lifting the Witten genus to the level of spectra.
It has been conjectured that this is only one of many projection maps decomposing \(MString\) into a direct sum of variants of \(tmf\) and spectra of lower chromatic flavor, just as the Atiyah-Bott-Shapiro orientation \(MSpin\to ko\) is the first projection of the Anderson-Brown-Peterson splitting of \(MSpin\) \cite{ABP_splitting}.
At primes  \(p\geq5\) \(MString\) is known to split into a sum of \(BP\)s by work of Hovey \cite{Hovey_MO8large}, while at the prime \(3\) ongoing work by Lorman, McTague, and Ravenel promises to provide a splitting analogous to Pengelleys \(2\)-local splitting of \(MSU\) \cite{Pengelley_MSU}. 
At the prime~\(2\) a candidate for the next summand occurring in \(MString\) is \(\Sigma^{16}tmf_0(3)\), see \cite[§~7]{MahowaldRezk_level3}, but not much is known in general.

A question that might be easier to answer is whether the Ando-Hopkins-Rezk orientation admits a section.
It is known to be surjective in homotopy by unpublished work of Hopkins and Mahowald, a proof being presented by Devalapurkar in \cite{devalapurkar_ahr}.
In a more recent paper, Devalapurkar also reduces the question of a section, at \(p=2\), to more classical conjectures in homotopy theory and centrality statements about Ravenel's \(X(n)\) spectra \cite{Devalapurkar_higherThom}.
As in the Anderson-Brown-Peterson splitting, at least after inverting \(6\), it seems unlikely that a section could have more structure than being unital, by work of McTague \cite{McTague_notQuotient}.

At \(p=2\), one might gain some understanding by focusing on one chromatic layer at a time.
Using the \(K(1)\)-local equivalence \(MString\to MSpin\) one could use the known decompositions of \(L_{K(1)}MSpin\) and \(L_{K(1)}tmf\) into copies of \(KO\) to study this question at height \(1\).
By work of Laures \cite{split6} and Hopkins \cite{Hopkins_K1Einfty} we know explicit presentations of \(L_{K(1)}MSpin\), \(L_{K(1)}tmf\), and \(\tau_W\) in the category of \(K(1)\)-local \(\einfty\)-rings.
These show that a section can not be made \(\einfty\).

In this paper we make some progress towards understanding the situation in the \(K(2)\)-local setting.
Using results of Hovey, Ravenel, and Sadofsky we are able to prove the following.
\begin{theoremalph}\label{premaintheorem}
    Let \(E=E(\ff_4,\widehat{C_0})\) be the Morava \(E\)-theory associated to the formal group of the supersingular elliptic curve
    \[C_0:y^2+y=x^3\]
    and let \(F_{3/2}\) be the open normal subgroup of automorphisms of the formal group law which are the identity up to order \(8\).
    Then there is a map
    \[D:MU\langle 6\rangle \to L_{K(2)}\bigoplus_I E\]
    such that its free \(E^{hF_{3/2}}\)-linearization induces a \(K(2)\)-local equivalence
    \[E^{hF_{3/2}}\otimes MU\langle 6\rangle\xrightarrow{\sim}\bigoplus_{I} E.\]
    Moreover, the map \(D\) can be chosen so that one projection is the sigma orientation \(\tau_\sigma\in E^0MU\langle 6\rangle\) while the others are reduced classes, that is they pair to zero with the unit.
\end{theoremalph}

In the case of \(MString\), by some kind of miracle\footnote{Having to do with the fact that, in \(K(2)\)-homology, the map \(K(\zz,3)\to BString\) {\it looks like} it should factor over the multiplication-by-\(2\) map on \(K(\zz,3)\).}, one can shrink the involved Galois extension of \(L_{K(2)}\ss\) down quite a lot\footnote{Strictly speaking, \(E^{hH}\) is not a Galois extension of \(L_{K(2)}\ss\) since \(H\) is not a normal subgroup of \(\gg\).}, and determine some of the projections.
\begin{theoremalph}\label{maintheorem}
    Consider the open subgroup \(H=\overline{\langle F_{3/2},\alpha, -1,\omega,\sigma\rangle}\subset \mathbb{G}\) of the Morava stabilizer group (see \cref{subsection:moravaGroup} for the definition of these elements).
    There is a map
    \[Q:BSpin_+\to L_{K(2)}\bigoplus_{j\in J} \Sigma^{n_j}TMF_0(3)\]
    with \(n_j\in\{0,16,32\}\) so that the \(E^{hH}\)-linear map
    \[E^{hH}\otimes MString\xrightarrow{\tau_Wpr^*Q}L_{K(2)}\bigoplus_{j\in J} \Sigma^{n_j}TMF_0(3)\]
    is a \(K(2)\)-local equivalence.

    Furthermore, one can choose \(Q\) such that the map 
    \[MString\xrightarrow{\tau_W\oplus\tau_W pr^*p}TMF\oplus\Sigma^{16}TMF_0(3)\]
    tensored with the identity on \(E^{hH}\) appears as a direct summand, where
    \[p=\pi_4/\Delta^6\in TMF_0(3)^{16}BSpin\]
    and we implicitly fix an \(E^{hH}\)-linear splitting of \(E^{hH}\otimes (TMF\oplus\Sigma^{16}TMF_0(3))\) as a sum of shifts of \(TMF_0(3)\).
    It thus admits an \(E^{hH}\)-linear section, and restricted to \(E^{hH}\otimes TMF\) this section can be made unital.
\end{theoremalph}
The statements in the abstract then follow from tensoring up this section to the Galois extension \(E^{hF_{3/2}}\), whose Galois group has order \(96\).

There are two directions in which one could try to push this result: firstly, one could try to make the section equivariant for the action of the Galois group, which would immediately give a section of the \(K(2)\)-local Ando-Hopkins-Rezk orientation.

Secondly, if one could show that all the \(Q_j=pr_j\circ Q\), in certain clusters, lift to spectra \(M_j\) such that \(E^{hH}\otimes M_j\simeq\oplus \Sigma^{?}TMF_0(3)\), one would have produced an additive \(K(2)\)-local splitting of \(MString\) into copies of the \(M_j\)s.
That is, one could try to 'divide' both the target and the map by \(E^{hH}\), and the 'furthermore' part of Theorem \ref{maintheorem} is a first step in this direction.
The main challenge in both approaches is to say anything meaningful about the \(Q_j\).
As we will see, one way of determining the \(Q_j\) involves the formal group law of the elliptic curve, which we know well, but the resulting formulas seem too complicated to be used directly.
Additionally, apart from \cite{tmf03} for \(TMF_0(3)\), there are not many calculations of the cohomology of \(BSpin\) for even the most natural candidates for the \(M_j\), such as \(TMF\) or \(TMF_0(5)\), making the second approach seem even harder.

% Looking towards the future, we know that the pattern mentioned above does not contiue any further.
% On the one hand, Adams' resolution of the Hopf invariant 1 problem directly shows that \(\mathcal{A}//\mathcal{A}(n)\) is not realizable as a spectrum for \(n\geq3\).
% On the other hand, Hovey \cite[prop.~2.3.4]{Hovey_NoHigherOrientations} showed that the Thom spectra of higher Whitehead covers of \(BO\) can't admit ring maps into certain fixed points of Morava \(E\)-theories at of height \(p-1\) at primes \(p\geq 5\).

Let us give a rough overview of the proof of Theorem \ref{maintheorem}.
It is well known (\cite{HoveyRavenel_7connected}) that there is a finite even spectrum \(Z\) such that \(Z\otimes MU\langle 6\rangle\) splits as a sum of \(BP\)s \(2\)-locally, so it splits as a sum of \(BP\langle 2\rangle\)s \(K(2)\)-locally (\cite{HoveySadofsky_Invertible}).
This allows us to construct a unital map \(E\to Z'\otimes MU\langle 6\rangle\), where \(Z'\) is a finite sum of shifts of \(Z\).
\(Z'\) comes with a preferred unital map to \(E\), and a concrete calculation shows that, in \(K\)-homology, this factors uniquely over \(E^{hF_{3/2}\ss}\to E\).

Putting this together we obtain a unital map of \(K_0E\)-comodules
\[K_0E\to K_0\nb{E^{hF_{3/2}\ss}\otimes MString}.\]
The Milnor-Moore argument then tells us that the target of this map must be a cofree comodule, from which we deduce that there is a \(K(2)\)-local equivalence
\[E^{hF_{3/2}\ss}\otimes MString\simeq \bigoplus E.\]
We then show, for suitable choices made in the Milnor-Moore argument, that we can force the above equivalence to be \(E^{hF_{3/2}\ss}\)-linear!
This is a key step that gives us a lot of control over the splitting.
In particular, the splitting is determined by its restriction to \(MString\), and thus by a set of classes in \(E^0BString\).

Studying the \(K_0E\)-comodule structure of \(K_0MString\) in more detail, we are then able to prove that these characteristic classes already factor over suspensions of \(TMF_0(3)\) and \(BSpin\), which leads us to the first statement of Theorem \ref{maintheorem}.
Finally, we prove the 'furthermore' part of Theorem \ref{maintheorem} by concrete calculations using the theory of cubical structures and cannibalistic classes.

\addcontentsline{toc}{subsection}{Notation and Conventions}
\subsection*{Notation and Conventions}
For \(R\) a homotopy ring spectrum, \(X\) some other spectrum, and classes \(a\in R_mX,\,b\in R^nX\) we denote by
\[\langle a,b\rangle\in \pi_{m-n}R\]
their Kronecker pairing.
If \(X\) is a suspension spectrum, we also denote by
\[a\smallfrown b\in R_{m-n}X\]
their cap product.
In the case that \(R=E\) is a Morava \(E\)-theory, we will also use this notation for the analogous pairings between \(E^*X\) and \(E^\vee_*X\).

Also, note that we often abbreviate the pointed suspension spectrum \(\Sigma_+^\infty X\) of a space \(X\) as \(X_+\) when there is little room for confusion.

Finally, while most parts of this paper take place in \(hSp\), the homotopy \(1\)-category of spectra, we sometimes employ constructions and concepts which are most naturally interpreted in the language of \(\infty\)-categories.
The standard references for this are \cite{Lurie_HTT} and \cite{Lurie_HA}, and we will use their content freely.

\addcontentsline{toc}{subsection}{Outline}
\subsection*{Outline}
In \cref{section:morava} we recall some facts about \(K(2)\)-local \(E\)-homology, the structure of the Morava stabilizer group, and of its action on \(E_*\), mostly following the exposition in \cite{algdual} and \cite{k2localmoore}, discuss the Hopf algebroid structure on \(\nb{K_0,K_0E}\), and set the stage for the Milnor-Moore argument.

In \cref{section:comodstruct} we explain the natural comodule structure of \(K_0MString\) and relate it to a twisted \(\mathbb{G}\)-action on \(K_0BSpin\).
We also explain how one can use the theory of cubical structures to determine this twisted action and some of its invariants.

In \cref{section:milnormoore} we put these pieces together:
The results of \cite{HoveySadofsky_Invertible} and \cite{HoveyRavenel_7connected} supply us with the prerequisites to apply the Milnor-Moore argument of \cite{mm} to finite Galois extensions of \(MU\langle6\rangle\) and \(MString\).
This proves \cref{premaintheorem} as \cref{thm:splitMU6} and a first approximation of \cref{maintheorem} as \cref{thm:splitMString1}

In \cref{section:projections}, we show that this splitting descends to the extension \(E^{hH}\), see \cref{thm:splitMString5}, and use our partial calculation of the invariants in \(K_0BSpin\) under the twisted action to deduce the 'furthermore' part of \cref{maintheorem}, see \cref{prop:finalSummand}.

Finally, we have banished several more technical discussions and lemmas to the Appendix.
\begin{itemize}
	\item Appendix \ref{section:minlorMoore} introduces and proves the variant of the Milnor-Moore argument we will use.
	\item Appendix \ref{section:iwasawaCoalg} studies coalgebras associated to certain \(p\)-adic analytic Lie groups, which include the restricted Morava stabilizer groups, and proves that they meet the requirements to apply the Milnor-Moore argument.
	\item Appendix \ref{section:contdual} contains some duality statements about Morava \(E\)-(co)homology, which we felt would only clutter the main part of this thesis.
\end{itemize}

\addcontentsline{toc}{subsection}{Acknowledgements}
\subsection*{Acknowledgements}
I want to thank Gerd Laures and Bj\"orn Schuster for helping me set foot in chromatic homotopy theory and always having an open ear for my questions.
Also, I want to thank Haynes Miller for inviting me to the MIT for the fall term of 2023, where a lot of the ideas of the present paper started to form.
Over the course of this project I enjoyed clarifying discussions with Keita Allen, Sanath Devalapurkar, Daniel Garrido, Mike Hopkins, Ishan Levy, Stephen McKean, Carl McTague, Piotr Pstr\c{a}gowski, and Laurent Smits.

This work was supported by the Deutsche Forschungsgemeinschaft - \(510812084\).

\section{The Morava stabilizer group and \(K_*E\)-comodules}\label{section:morava}
We adopt the notation of \cite{k2localmoore} for most things concerning \(E\), see sections 2 and 3 therein, but we drop the subscript \(C\) from the Morava stabilizer groups.

The Morava \(E\)-theory we consider is an elliptic spectrum \(E\) with coefficient ring
\[E_*=W(\ff_4)[\![u_1]\!][u^{\pm}],\,|u_1|=0,\,|u|=-2\]
and Weierstrass elliptic curve \(C:y^2+3u_1xy+(u_1^3-1)y=x^3\).
It carries the universal deformation of the formal group law associated to the supersingular curve
\[C_0:y^2+y=x^3\]
over \(\ff_4\).
Most importantly, since, in the coordinate \(t=-x/y\), the \(-2\)-series is given by \([-2](t)=t^4\), we can apply the Devinatz-Hopkins fixed point machinery to this variant of \(E_2\).

\subsection{The stabilizer group and generators of certain subgroups}\label{subsection:moravaGroup}

Fix a generator \(\zeta\) of \(\ff_4\) over \(\ff_2\), and denote a Teichmüller lift to the Witt vectors by the same letter.
The endomorphism ring of the formal group law \(F_{C_0}\) associated to \(C_0\) in the coordinate \(t=-x/y\) is given by
\[\mathrm{End}(F_{C_0})\simeq W(\ff_4)\langle T\rangle/(T^2+2,\zeta T-T\zeta^2)\]
where \(\zeta(t)=\zeta t\) and \(T(t)=t^2\).
Then every endomorphism can be written uniquely as a power series
\[\gamma=\sum_{n\geq0}a_nT^n\]
with \(a_n\in\{0,1,\zeta,\zeta^2\}\).

The restricted Morava stabilizer group is given by the units \(\mathbb{S}=\mathrm{End}(F_C)^\times\) and the full Morava stabilizer group is its extension by the Galois group
\[\mathbb{G}=\mathrm{End}(F_C)^\times\rtimes\mathrm{Gal}(\ff_4/\ff_2).\]
\begin{remark}
  When viewing \(\zeta\in W(\ff_4)^\times\subset\mathbb{G}\) as an element of the Morava stabilizer group we will often denote it by \(\omega\) to make the distinction between it and the element \(\zeta\in E_0=W(\ff_4)[\![u_1]\!]\) clearer.
\end{remark}
Before we continue let us recall some definitions and theorems from the theory of profinite groups.
\begin{definition}\,
  \begin{itemize}
    \item A topological group \(G\) is called {\it profinite} if it is compact, Hausdorff, and the set of normal open subgroups \(U\subset G\) is a basis of neighborhoods of \(e\in G\).
    \item A profinite group \(G\) is called {\it finitely generated} if there exists a finite subset \(S\subset G\) such that \(G\) is the closure of the subgroup generated by \(S\).
    \item A profinite group \(G\) is called a {\it pro-\(p\) group} if every open subgroup \(U\subset G\) has index a power of \(p\): \([G:U]=p^i\) for some \(i\).
    \item A pro-\(p\) group \(G\) is called {\it powerful} if \([G,G]\subset \overline{\langle g^p\mid g\in G\rangle}\) for \(p\) odd, while for \(p=2\) we require \([G,G]\subset \overline{\langle g^4\mid g\in G\rangle}\).
    \item The {\it lower \(p\)-series} of a pro-\(p\) group \(G\) is the filtration given by \(P_1(G)=G\) and \(P_{i+1}(G)=\overline{P_i(G)^p[P_i(G),G]}\).
    \item A pro-\(p\) group is called {\it uniform} if it is finitely generated, powerful, and for all \(i\geq1\) we have \([P_i(G):P_{i+1}(G)]=[P_{i+1}(G):P_{i+2}(G)]\).
    \item Finally, we call a profinite group \(G\) a {\it compact \(p\)-adic analytic Lie group} if it has an open normal subgroup which is a uniform pro-\(p\) group of finite index in \(G\).
  \end{itemize}
\end{definition}
\begin{theorem}[{\cite[Theorem~4.5]{analyticProP}}]
  Let \(G\) be a finitely generated powerful pro-\(p\) group.
  Then \(G\) is uniform if and only if \(G\) is torsion free.
\end{theorem}
\begin{theorem}[{\cite[Theorem~4.9]{analyticProP}}]
  Let \(G\) be a uniform pro-\(p\) group with topological generating set \(\{g_1,\ldots,g_d\}\) of minimal cardinality.
  Then the map
  \[\zz_p^d\to G,\,(n_1,\ldots,n_d)\mapsto g_1^{n_1}\ldots g_d^{n_d}\]
  is a homeomorphism.
\end{theorem}
  \begin{theorem}[\cite{Lazard_AnalyticP},\cite{NikolovSegal_OpenFiniteIndex}]\label{thm:openSubgroups}
	Let \(G\) be a compact \(p\)-adic analytic Lie group.
	Then a subgroup \(H\subset G\) is open iff it has finite index.
	Moreover, the forgetful functor from compact \(p\)-adic analytic Lie groups and locally analytic homomorphisms to groups is fully faithful.
  \end{theorem}

\begin{remark}
  The group \(\mathbb{S}\) has a natural structure of a profinite group with the subspace topology with respect to the \(t\)-adic topology on \(\ff_4[\![t]\!]\).
  This also equips \(\mathbb{G}\) with the structure of a profinite group.
\end{remark}
Denote by \(\sigma\) the generator of the Galois group, and for \(\gamma\in \mathbb{S}\) let \(\overline{\gamma}=\sigma\gamma\sigma\).
There are two more or less natural structures on \(\mathbb{S}\) which will play an important role later on.
\begin{definition}
  Let \(\gamma=x+yT\) be an endomorphism, with \(x,y\in W(\ff_4)\).
  The {\it determinant} of \(\gamma\) is defined to be
  \[|\gamma|=x\overline{x}+2y\overline{y}.\]
  This extends to a morphism of groups \(\mathbb{G}\to\zz_2^\times\) via \(|\sigma|=1\).
\end{definition}

\begin{definition}
  Let \(F_{n/2}\mathbb{S}\subset\mathbb{S}\) be the normal subgroup of automorphisms of the form \(1+xT^n\).
  We will often shorten the notation to just \(F_{n/2}\).
\end{definition}

Consider the elements
\[\pi=1+2\zeta\text{ and }\alpha=(1-2\zeta)/\sqrt{-7}\]
with \(\sqrt{-7}\equiv1\) mod \((4)\).
Their determinants are \(3\) and \(-1\), showing that the determinant is a surjective map \(\mathbb{G}\to\zz_2^\times\).
We have that \(\pi^2,\alpha^2\), and \(\pi\alpha\) all lie in \(F_{3/2}\).
Let \(G_{24}\) be the group of automorphisms of the elliptic curve \(C_0\) over \(\ff_4\) preserving the base point, which maps injectively into the automorphisms of the associated formal group law, and let \(G_{48}\) be the extension of \(G_{24}\) by the Galois group.
The group \(G_{24}\) is isomorphic to \(Q_8\rtimes\ff_4^\times\), with generators \(i\) and \(\omega\), and we let \(j=\omega i\omega^2,k=\omega j\omega^2\).
Explicitly, in Weierstrass coordinates \(i\) acts as
\[i^*x=x+1,\,i^*y=y+x+\zeta.\]
Then \(\mathbb{S}\) is given as
\[\mathbb{S}=\overline{\langle F_{3/2},\pi\rangle}\rtimes G_{24}\]
where the overline denotes closure in the natural profinite topology.
All elements of \(G_{24}\) have determinant \(1\).
We will later have use for the following statement:

\begin{lemma}\label{lem:generatorsF32prime}
  The group \(F_{3/2}'=\{\gamma\in F_{3/2}\mid |\gamma|=1\}\) is topologically generated by \(\alpha^2\), \([i,\alpha]\), and \([j,\alpha]\).
\end{lemma}

\begin{proof}[Proof of \cref{lem:generatorsF32prime}]
  In this proof we will implicitly use the isomorphism of endomorphism algebras of the elliptic and the Honda formal group laws described in Lemma 3.1.2 of \cite{k2localmoore}.

  By Lemma 2.2.1 and the proof of Proposition 2.5.3 in \cite{algdual} we know that \(F_{4/2}\) is a powerful pro-2 group topologically generated by any set of elements generating all of \(F_{4/2}/F_{6/2}\).
  The map
  \[\ff_4\times\ff_4\to F_{4/2}/F_{6/2},(a,b)\mapsto [1+aS^4+bS^5]\]
  is an isomorphism of groups.
  Direct computation shows that a minimal set of generators of \(F_{4/2}/F_{6/2}\), and thus of \(F_{4/2}\), is given by
  \[\alpha^2,[i,\alpha]^2,[j,\alpha]^2,\pi/\alpha.\]
  
  \(F_{4/2}\) is also uniform:
  Since we already know it to be finitely generated powerful we only need to show that it is free of torsion by \cite[Theorem~4.5]{analyticProP}.
  For \(n\geq 4\) consider an element in \(F_{4/2}\) of the form
  \[1+aS^n+xS^{n+1}\]
  with \(a\in\{1,\zeta,\zeta^2\}\).
  Its square is given by
  \[1+aS^{n+2}+xS^{n+3}+aa^{\sigma^n}S^{2n}+xx^{\sigma^{n+1}}S^{2n+2}+(ax^{\sigma^n}+xa^{\sigma^{n+1}})S^{2n+1}.\]
  For this to be zero we would need \(a=0\), a contradiction.

  By \cite[Theorem~4.9]{analyticProP} any element of \(F_{4/2}\) can now uniquely be written as
  \[\alpha^{2a}[i,\alpha]^{2b}[j,\alpha]^{2c}(\pi/\alpha)^d\]
  for \(a,b,c,d\in\zz_2\), which has determinant \((-3)^d\).
  Thus, \(F_{4/2}'=\overline{\langle \alpha^2,[i,\alpha]^2,[j,\alpha]^2   \rangle}\).
  Finally, the quotient \(F_{3/2}'/F_{4/2}'=F_{3/2}/F_{4/2}\) is generated by \([i,\alpha]\) and \([j,\alpha]\), whence the claim.
\end{proof}

\subsection{Fixed points of Morava \(E\)-theory}
By work of Goerss, Hopkins, and Miller the group \(\mathbb{G}\) acts on \(E\) through \(E_\infty\)-ring maps.
In fact, the space of \(E_\infty\)-ring endomorphisms is isomorphic to \(\mathbb{G}\) as a discrete set, see \cite[§~7]{GoerssHopkins_Structured}.
\begin{notation}
  For \(g\in\mathbb{G}\) we will denote the associated map by \(\psi^g:E\to E\).
\end{notation}
In \cite{DevinatzHopkins_FixedPoints} Devinatz and Hopkins develop a theory for taking homotopy fixed points of \(E\) with respect to any closed subgroup of \(\mathbb{G}\) which agrees with the usual notion if the subgroup is finite.
Note that, while they only develop this for the Honda formal group law of height \(n\) over \(\ff_{p^n}\), their approach works whenever every automorphism of the formal group law over the algebraic closure is already defined over the field one is working with.
This is the case for us as \([-2](t)=t^4\) and thus any automorphism of \(F_{C_0}\) over \(\overline{\ff_4}\) must be fixed by the square of the Frobenius, and thus must already be defined over \(\ff_4\).
Let us collect what we need about these fixed point spectra in the following theorem.

\begin{theorem}[{\cite[Theorem~1]{DevinatzHopkins_FixedPoints},\cite[Theorems~5.4.4,~5.4.9]{Rognes_galois}}]\,
  \begin{enumerate}
    \item The unit map \(\mathbb{S}\to E^{h\mathbb{G}}\) is a \(K(n)\)-localization.
    \item For each pair of closed subgroups \(H\subset K\subset\mathbb{G}\) with \(H\) normal and of finite index in \(K\) the map \(E^{hH}\to E^{hK}\) is a faithful \(K(n)\)-local Galois extension for the group \(K/H\).
    \item For any closed subgroup \(H\subset\mathbb{G}\) the map
    \[E_*^\vee E^{hH}\to C^0\nb{\mathbb{G}/H,E_*},\,x\mapsto \nb{[g]\mapsto \langle x,\psi^g\circ\iota\rangle},\]
	where \(\iota: E^{hH}\to E\) is the inclusion of the fixed points and \(E_*\) carries the \(\mm\)-adic topology, is well-defined and an isomorphism.
	\item For each closed subgroup \(H\subset\gg\) the spectrum \(E^{hH}\) is \(K\)-local.
  \end{enumerate}
\end{theorem}

Let us discuss how various fixed points of \(E\) relate to \(K\)-localizations of \(TMF\) with level structure.
We include this discussion since we find the available sources stating the content of Proposition \ref{cor:EfixedTMF} not very elucidating, but claim no originality.
In fact, the following was already known to Hopkins and Mahowald.

Let \(R\) be the ring
\[R:=\frac{\zz\sb{\frac{1}{3},B,C,\Delta^{-1}}}{\nb{B^3=(B+C)^3}},\]
with \(\Delta\) the discriminant of the  Weierstrass curve
\[C_3:y^2+(3C-1)xy+(-3C^2-B-3BC)=x^3\]
and let \(P\) and \(Q\) be the points
\[P=(0,0),\,Q=(C,B+C).\]
Then \((P,Q)\) is a full level \(3\) structure on \(C_3\), and the classifying map
\[(C_3,P,Q):\mathrm{Spec}(R)\to \mathcal{M}(3)\]
is an equivalence, see \cite[§~2.2.10]{KatzMazur_Moduli}.
This gives us that \(TMF(3)\) is even periodic and Landweber exact, and yields preferred isomorphisms
\[\pi_0TMF(3)\simeq R\text{ and }G^Q_{TMF(3)}\simeq\widehat{C_3}.\]

The points \(P_0=(0,0)\) and \(Q_0=(1,\zeta)\) equip \(C_0\) with a full level \(3\) structure, and the associated map \(R\to \ff_4\) sends \(C\) to \(1\) and \(B\) to \(\zeta^2\).

Since the map \(\mathcal{M}(3)\to\mathcal{M}_{Ell}[\frac{1}{3}]\) is finite \'etale, the space of dotted arrows in the diagram
% https://q.uiver.app/#q=WzAsNCxbMCwwLCJcXG1hdGhybXtTcGVjfShcXGZmXzQpIl0sWzEsMCwiXFxtYXRoY2Fse019KDMpIl0sWzEsMSwiXFxtYXRoY2Fse019X3tFbGx9W1xcZnJhY3sxfXszfV0iXSxbMCwxLCJcXG1hdGhybXtTcGZ9KEVfMCxcXG1tKSJdLFswLDNdLFswLDEsIihDXzAsUF8wLFFfMCkiXSxbMSwyXSxbMywyLCJDX0UiLDJdLFszLDEsIiIsMix7InN0eWxlIjp7ImJvZHkiOnsibmFtZSI6ImRhc2hlZCJ9fX1dXQ==
\[\begin{tikzcd}
	{\mathrm{Spec}(\ff_4)} & {\mathcal{M}(3)} \\
	{\mathrm{Spf}(E_0,\mm)} & {\mathcal{M}_{Ell}[\frac{1}{3}]}
	\arrow["{(C_0,P_0,Q_0)}", from=1-1, to=1-2]
	\arrow[from=1-1, to=2-1]
	\arrow[from=1-2, to=2-2]
	\arrow[dashed, from=2-1, to=1-2]
	\arrow["{C_E}"', from=2-1, to=2-2]
\end{tikzcd}\]
is contractible.
Denote the associated full level 3 structure on \(C_E\) by \((P_E,Q_E)\).
\begin{proposition}
	The map
	\[TMF(3)\to E\]
	associated to the above full level 3 structure is a \(K\)-equivalence.
\end{proposition}
\begin{proof}
	Let us first determine \(\pi_*L_K TMF(3)\).
	As \(TMF(3)\) is even periodic and Landweber exact, we have
	\[\pi_*L_K TMF(3)\simeq (v_2u^3)^{-1}R^\wedge_{(2,v_1u)}[u^{\pm}]\]
	for some choice of unit \(u\in \pi_{-2}TMF(3)\).
	Using the coordinate \(-x/y\) on \(C_3\) we find
	\[uv_1=3C-1\text{ and }u^3v_2=-3C^2-B-3BC.\]
	Mod \((2,uv_1)\) both \(u^3v_2\) and \(\Delta\) are congruent to \(1\), so we can leave out the localizations.
	Setting \(B'=B/(uv_1+1)\) we find
	\[\pi_0L_KTMF(3)=\nb{\zz_2^\wedge[B']/(B'^2+B'+1/3)}\pb{uv_1}.\]
	By Hensel's lemma there is a unique \(b\in W(\ff_4)\subset W(\ff_4)\pb{u_1}\) such that
	\[b^2+b+1/3=0\text{ and }b\equiv\zeta^2\text{ mod }\mm,\]
	so the map to \(\pi_0E\) must send \(B'\) to \(b\).
	In particular, since \(b\) agrees with a Teichmüller lift of \(\zeta\) up to multiples of \(2\), the map to \(E\) induces an equivalence
	\[\zz_2^\wedge[B']/(B'^2+B'+1/3)\xrightarrow{\sim}W(\ff_4).\]
	Finally, since \(TMF(3)\to E\) is a ring map, the image of \(uv_1\) agrees with the Hazewinkel generator \(3u_1\) up to multiples of \(2\), so the map is an equivalence.
\end{proof}
There is a natural action of \(GL_2(\zz/3)\) on \(\mathcal{M}(3)\), given by:
\[\begin{pmatrix} a & b\\ c& d\end{pmatrix}.(C,P,Q)=(C,aP+bQ,cP+dQ).\]

By functoriality this acts on \(TMF(3)\) by \(\einfty\)-ring maps, and so by the previous Proposition also on \(E\).
Since \(CAlg(E,E)\) is the discrete set \(\gg\), this action is completely determined by a group homomorphism
\[GL_2(\zz/3)\to \gg.\]

The abelian group \(C_0(\ff_4)\) is isomorphic to \(\zz/3\times\zz/3\) on the generators \(P_0\) and \(Q_0\), and the action of the \(G_{48}\) on \((\ff_4,C_0)\) induces an isomorphism
\[G_{48}\simeq GL_2(\zz/3).\]
The map \(\mathrm{Spec}(\ff_4)\to\mathcal{M}(3)\) of stacks over \(\mathcal{M}_{Ell}[\frac{1}{3}]\) classifying \((C_0,P_0,Q_0)\) is equivariant with respect to these actions.
This shows that the map \(GL_2(\zz/3)\to\gg\) is given by
\[GL_2(\zz/3)\simeq G_{48}\subset \gg,\]
where the last map is the inclusion discussed in the previous Section.
\begin{proposition}\label{cor:EfixedTMF}
	The above discussion provides canonical equivalences
	\begin{itemize}
		\item \(L_KTMF(3)\simeq E\),
		\item \(L_K TMF_1(3)\simeq E^{h\langle\omega,\sigma\rangle}\),
		\item \(L_KTMF_0(3)\simeq E^{h\langle -1,\omega,\sigma\rangle}\), and
		\item \(L_KTMF\simeq E^{hG_{48}}\).
	\end{itemize}
\end{proposition}
\begin{proof}
	An explanation of the equivalences
	\[TMF_1(3)\simeq TMF(3)^{h\langle\omega,\sigma\rangle}, TMF_0(3)\simeq TMF(3)^{h\langle-1,\omega,\sigma\rangle},\text{ and }TMF[1/3]\simeq TMF(3)^{hG_{48}}\]
	can be found in the introduction of \cite{MahowaldRezk_level3}.

	Let \(G\subset GL_2(\zz/3)\) be a subgroup.
	We only need to show that the natural map
	\[L_K\nb{TMF(3)^{hG}}\to\nb{L_KTMF(3)}^{hG}\]
	is an equivalence.
	By \cite[prop.~7.10(e)]{666} it suffices to show that \(TMF(3)_2^\wedge\) is \(E(2)\)-local, since on such spectra \(L_K\) is computed by a limit, so that it commutes with limits.

	From \cite[thm.~1.1]{LawsonNaumann_tmf13BP2} we know that \(tmf_1(3)_{(2)}\) is a form of \(BP\langle 2\rangle\), so that \(v_2^{-1}tmf_1(3)_{(2)}\) is a form of \(E(2)\).
	Note that the Hazewinkel \(v_2\) is a unit in \(TMF(3)_*\) since it divides the discriminant of \(C_3\), so that we get a map of \(\einfty\)-rings
	\[v_2^{-1}tmf_1(3)_{(2)}\to TMF(3)_2^\wedge.\]
	So \(TMF(3)_2^\wedge\) is an \(E(2)\)-module, and thus \(E(2)\)-local.
\end{proof}

\subsection{Hopf algebroid structures}
Let \(\mm=(2,u_1)\subset E_0\) be the maximal ideal.
For well-behaved spectra \(X\), for example if \(K_*X\) is concentrated in even degrees, one has
\[K_*X=\nb{E_*X}/\mm.\]
Since the action of \(\mathbb{G}\) on \(E_*\) preserves the maximal ideal, in these cases \(K_*X\) is a \(\mathbb{G}\)-module, natural in \(X\).
Said differently, in these cases \(K_0X\) is naturally a comodule over the Hopf algebroid \(\nb{K_0, (E_0E)/\mm=K_0E}\), which is isomorphic to \(\nb{\ff_4, C^0\nb{\mathbb{G},\ff_4}}\).
Let us sketch how Hopf algebroid structure using this identification.
\begin{remark}
  When working with these kinds of Hopf algebroids, we will often use functions of the form
  \[\delta_g:\mathbb{G}\to\ff_4,\,h\mapsto \begin{cases} 1 & g=h\\ 0 & \text{else}\end{cases}.\]
  While these are clearly not continuous functions, the infinite sums we write down will be, so this is just a useful notation.
\end{remark}

\begin{proposition}\label{prop:hopfstruct}
  Let \(X\) be a spectrum such that \(K_*X\) is concentrated in even degrees.
  Then the coaction map is given by
  \[\psi:K_0X\to C^0(\mathbb{G},\ff_4)\otimes_{\ff_4}K_0X,\,a\mapsto \sum_{g}\delta_g\otimes (g^{-1}).a.\]
  In particular, the diagonal is given by
  \[\Delta(\delta_g)=\sum_{xy=g}\delta_x\otimes\delta_y.\]
  The left and right unit are given by
  \[\eta_L(a)(g)=a\text{ and }\eta_R(a)(g)=g.a,\]
  the counit is given by evaluation at \(1\in\mathbb{G}\), and the conjugation by
  \[(cf)(g)=g.(f(g^{-1})).\]
\end{proposition}
\begin{proof}
  The claims above follow from the isomorphism
  \[E^\vee_*E\xrightarrow{\sim}C^0(\mathbb{G},E_*),\,a\mapsto\nb{g\mapsto\langle a, \psi^g\rangle}\]
  and transporting the Hopf algebroid structure along it.
\end{proof}

Call the above Hopf algebroid \(\Sigma\) and consider the Hopf algebra
\[\Sigma'=\nb{\ff_4,C^0\nb{\mathbb{S},\ff_4}}\]
together with the restriction map \(\Sigma\to\Sigma'\).
One can check that \(\Sigma'\) is the associated Hopf algebra as in \cite[Definition~A1.1.9]{Ravenel_green}.
\begin{proposition}\label{prop:comodulestruct}\label{prop:twisted=involution}
  The category of (left) \(\Sigma\)-comodules is equivalent to the category of (left) \(\Sigma'\)-comodules \(M\) with an \(\ff_4\)-antilinear involution \(\tau\) intertwining the \(\mathbb{S}\)-action up to conjugation:
  \[\psi\circ\tau=\nb{(\sigma_*\circ\mathrm{Ad}_{\sigma}^*)\otimes\tau}\circ\psi.\]
  The equivalence is implemented via the corestriction along \(\Sigma\to\Sigma'\) and setting \(\tau\) to be the map
  \[\tau=\nb{\sigma\circ\mathrm{ev}_\sigma\otimes\id}\circ\psi_\Sigma.\]
  Under this equivalence the comodule \(\Sigma\) is sent to
  \[\Sigma'\otimes_{\ff_2}\ff_4,\,\tau(f\otimes x)=\nb{\gamma\mapsto \overline{f(\overline{\gamma})}}\otimes x\]
  via the isomorphism
  \[\Sigma'\otimes_{\ff_2}\ff_4\to\Sigma,\,\delta_\gamma\otimes x\mapsto \delta_\gamma x+\delta_{\gamma\sigma}\bar{x}.\]
  The map \(K_0E^{h\mathrm{Gal}}\to K_0E=\Sigma\to\Sigma'\) gives a \(\Sigma'\)-comodule algebra isomorphism, and under this identification \(\tau\) acts on \(\Sigma'\) as \((\tau f)(\gamma)=\overline{f(\overline{\gamma})}\).
\end{proposition}
\begin{proof}
  Given a \(\Sigma'\)-comodule \(M\) together with \(\tau\) we can give it the structure of a \(\Sigma\) comodule by setting
  \[\psi_\Sigma(m)=\sum_{\gamma\in\mathbb{S}}\delta_{\gamma}\otimes \gamma^{-1}.m+\delta_{\sigma\gamma}\otimes\gamma^{-1}.\tau(m).\]
  This gives an inverse functor, so the claimed equivalence follows.

  The first claimed isomorphism follows immediately, while the second follows from the description \(E^\vee_*E^{h\mathrm{Gal}}=C^0(\mathbb{G}/\mathrm{Gal},E_*)\).
\end{proof}

\begin{remark}
  The categories of \(\Sigma\) and \(\Sigma'\)-comodules are equivalent to the categories of {\it discrete} \(\mathbb{G}\) and \(\mathbb{S}\)-representations over \(\ff_4\).
  This is very useful to keep in the back of one's mind when analyzing these objects, and is also the motivation for the Proposition above.
\end{remark}

Thus, to determine if a \(\Sigma\)-comodule is cofree, we may first show that the associated \(\Sigma'\)-comodule is and then try to choose a splitting compatible with the action of \(\tau\).
Since \(\Sigma'\) is a pointed coalgebra with exhaustive coradical filtration we have a form of the Milnor-Moore theorem available to determine whether a comodule algebra over it is cofree, see \cite{mm}.

We shall later need the following description of the group-like elements in \(\Sigma'\):
\begin{proposition}\label{prop:grouplikes}
  The group-like elements of \(\Sigma'=C^0\nb{\mathbb{S},\ff_4}\) are given by the group homomorphisms.
  Their \(\ff_4\)-linear span  equals the subset of functions factoring as
  \[\mathbb{S}\to\mathbb{S}/\nb{\overline{\langle F_{3/2},\pi\rangle}\rtimes Q_8}=\ff_4^\times\to\ff_4,\]
  that is it is equal to \(C^0(\ff_4^\times,\ff_4)\).
\end{proposition}

Let us also mention how to determine \(1\)-primitives of \(\Sigma'\)-comodules:
\begin{proposition}\label{prop:primInv}
  Let \(M\) be a (left) \(\Sigma'\)-comodule, and let
  \[P_1(M):=\{m\in M\mid \psi(m)=1\otimes m\}\]
  be the subspace of \(1\)-primitive elements.
  Then, viewing \(M\) as a discrete \(\mathbb{S}\)-representation, we find
  \[P_1(M)=M^{\mathbb{S}}.\]
\end{proposition}
\begin{proof}
  This follows from the equation
  \[\gamma^{-1}.m=(ev_\gamma\otimes\id)(\psi(m))\]
  and the injection \(\Sigma'\otimes M\to \mathrm{Set}(\mathbb{S},M)\).
\end{proof}

Let us now introduce the Milnor-Moore argument we will use throughout.
For a lengthier discussion and proofs, see Appendix \ref{section:minlorMoore}.
\begin{definition}
	Let \(\ff\) be a finite field and \(C\) be a Hopf algebra over \(\ff\) which, as a coalgebra, is pointed.
	Denote by \(G\) the subset of group like elements in \(C\).
	Let \(M\) be an \(\ff[G]\)-module in \(C\)-comodules.
	Consider the natural maps
	\[q_n:F_n(M)/F_{n-1}(M)\to F_n(C)/F_{n-1}(C)\otimes P_1(M),\,[m]\mapsto \sum_{i,g}[c_i]\otimes g^{-1}.m_{i,g}\]
	for \(\psi(m)\equiv \sum_ic_i\otimes m_i\) mod \(F_{n-1}(C)\otimes F_n(M)\) with \(c_i\in F_n(C)\) and \(m_i\in F_0(M)\).
	We call \(M\) {\it splittable} if these are surjective for all \(n\geq 0\).
\end{definition}
\begin{theorem}
	Let \(C\) be a Hopf algebra over \(\ff\) which, as a coalgebra, is pointed, and let \(G\subset C\) be the set of group-like elements.
	Let \(M\) be an \(\ff[G]\)-module in \(C\)-comodules which is splittable.
	Then, for any \(\ff\)-linear retraction \(r:M\to \ff[G]\otimes P_1(M)\) of the natural map \(\ff[G]\otimes P_1(M)\simeq F_0(M)\subset M\), the map
	\[h_r:M\xrightarrow{\psi}C\otimes M\xrightarrow{\id\otimes r}C\otimes\ff[G]\otimes P_1(M)\xrightarrow{\id\otimes\epsilon\otimes\id}C\otimes P_1(M)\]
	is an isomorphism of \(C\)-comodules and exhibits \(M\) as a (co)free \(\ff[G]\)-module in \(C\)-comodules.
\end{theorem}
\begin{proposition}
	Let \(M\) be an \(\ff[G]\)-algebra in \(C\)-comodules, and assume there exists a map \(s:C\to M\) of comodules with \(s(1)=1\).
	Then M is splittable.
\end{proposition}

\subsection{A factorization lemma}
In this Section we want to show that the \(\Sigma'\)-comodules
\[K_0\nb{E^{hF_{3/2}}\otimes MU\langle 6\rangle}\text{ and }K_0\nb{E^{hF_{3/2}}\otimes MString}\]
are splittable, see Definition \ref{def:splittable}.
As we show in Proposition \ref{prop:splittableCriterion} a sufficient condition is to admit a unital map of \(\Sigma'\)-comodules from \(\Sigma'\).
\begin{remark}
	For the definition of a splittable comodule to apply, we need them to be \(\ff_4[G]\)-modules, where \(G\) is the set of group like elements in \(\Sigma'\).
	In our case, group like elements correspond to group homomorphisms \(\ss\to \ff_4^\times\), all of which must contain \(F_{3/2}\) in their kernel, so that already \(K_0E^{hF_{3/2}}\) is an \(\ff_4[G]\)-module in \(\Sigma'\)-comodules.
\end{remark}

Let \(DA(1)\) be the even \(2\)-local finite complex called \(Z\) in \cite{HoveyRavenel_7connected}, see also \cite[§~4.1]{Mathew_tmfHomology}, and choose a 'unit' map \(\mathbb{S}\to DA(1)\) inducing an isomorphism in zeroth homotopy.

%To this end we will use the splitting of Hovey and Ravenel \cite[cor.~2.2]{HoveyRavenel_7connected} to produce a spectrum level section of \(DA(1)\otimes MU\langle 6\rangle\to E\), where \(DA(1)\) is the finite \(2\)-local complex called \(Z\) in \cite{HoveyRavenel_7connected}, and then show that, upon applying \(K_0\), this map factors through \(K_0\nb{E^{hF_{3/2}}\otimes MString}\).

\begin{lemma}\label{lem:section1}
  The map \(DA(1)\otimes MU\langle 6\rangle\to BP\to tmf_1(3)\) admits a unital section after \(K(2)\)-localization.
\end{lemma}
\begin{proof}
  The first part of the composite is the external product of any fixed choice of unital map \(DA(1)\to BP\) and of the canonical map \(MU\langle 6\rangle \to MU \to BP\), the second is the \(2\)-typicalization of the standard complex orientation of \(tmf_1(3)\).\footnote{This orientation maps along the canonical map \(tmf_1(3)\to E\) to \(uz_E\).}
  
  By \cite[cor.~2.2]{HoveyRavenel_7connected} we have that \(DA(1)\otimes MU\langle 6\rangle\) splits \(2\)-locally as a direct sum of suspensions of \(BP\).
  Inspecting the proof of \cite[thm.~2.1]{HoveyRavenel_7connected} we see that this splitting is obtained by choosing maps
  \[x_i:DA(1)\otimes MU\langle6\rangle\to \Sigma^{|x_i|}\ff_2\]
  forming a basis of \(\ff_2^*\nb{DA(1)\otimes MU\langle 6\rangle}\) over \(\ff_2^*BP\) and lifting them through the ring map \[BP\to \ff_2\] (using that \(DA(1)\otimes MU\langle 6\rangle\) is even).
  This shows that the first map in the composite can appear as one of the \(2\)-local projection maps, and thus admits a \(2\)-local section.
  Since the map \[DA(1)\otimes MU\langle 6\rangle \to BP\] induces an isomorphism on \(\pi_0\), the section can be chosen to be unital.

  Following \cite[thm.~4.1]{HoveySadofsky_Invertible} and the remark after its proof, and using that the \(2\)-typical complex orientation \(BP\to tmf_1(3)\) exhibits \(tmf_1(3)\) as a form of \(BP\langle 2\rangle\) (see \cite[thm.~1.1]{LawsonNaumann_tmf13BP2}), we find that the second map in the composite admits a unital section after \(K(2)\)-localization.
\end{proof}
\begin{lemma}\label{lem:section2}
  The map \(\nb{\mathbb{S}^0\oplus\mathbb{S}^0}\otimes\nb{\mathbb{S}^0\oplus\mathbb{S}^{-2}\oplus\mathbb{S}^{-4}}\otimes DA(1)\otimes MU\langle6\rangle \to E\), given by the external product of the map from Lemma \ref{lem:section1} and the map sending the spheres to the elements \(1,\zeta,1,u,u^2\) in \(\pi_* E\),  admits a \(K(2)\)-local unital section.
\end{lemma}
\begin{proof}
  In the \(K(2)\)-local category we have an equivalence
  \[\nb{\mathbb{S}^0\oplus\mathbb{S}^0}\otimes\nb{\mathbb{S}^0\oplus\mathbb{S}^{-2}\oplus\mathbb{S}^{-4}}\otimes tmf_1(3)\to E\]
  given by sending the spheres to the elements \(1,\zeta,1,u,u^2\in\pi_* E\)  (in that order).
  We can now take the section of Lemma \ref{lem:section1} and tensor it with the identity on the spheres.
\end{proof}

%Consider the open subgroup \(U=F_{3/2}\rtimes\mathrm{Gal}\) of \(\mathbb{G}\).
%The inclusion \(E^{hF_{3/2}}\to E\) induces the injection
%\[C^0(\mathbb{G}/F_{3/2},K_*)\to C^0(\mathbb{G},K_*)\]
%in \(K\)-homology.
%We want to show that, in \(K_0\), the map \(T_2\to E^{h\mathrm{Gal}}\) considered in \cite{mm} factors (necessarily uniquely) over \(E^{hF_{3/2}}\).

%Let us first explain what \(T_2\) is.
%At the prime \(2\), just like \(MU\) splits into suspensions of \(BP\), Ravenels spectrum \(X(4)\) splits into suspensions of a spectrum called \(T(2)\), see \cite[thm.~3.1.5]{Devalapurkar_higherThom}.
%It admits the structure of an \(E_2\)-algebra (\cite[rmk.~3.1.9]{Devalapurkar_higherThom}), and the inclusion \(T(2)\to BP\) induces an isomorphism in \(\pi_k\) for \(k\leq 12\).
%In particular, there is a unique elements in \(\pi_6T(2)\) mapping to the Hazewinkel \(v_2\) in \(BP\), which we will denote by the same letter.
%The spectrum \(T_2\) is now obtained by adjoining an element \(p\) in degree \(-2\) whose cube is an inverse of \(v_2\):
%\[T_2=\frac{T(2)[p]}{(p^3v_2-1)}=\nb{\mathbb{S}^0\oplus\mathbb{S}^2\oplus\mathbb{S}^4}\otimes \colim\nb{T(2)\xrightarrow{v_2}\Omega^6 T(2)\xrightarrow{v_2}\ldots}.\]
%Since \(\pi_{-2}E^{h\mathrm{Gal}}\) also contains elements whose cubes are inverses of \(v_2\) we see that the map
%\[T(2)\to BP\to MU\to E^{h\mathrm{Gal}}\]
%extends to \(T_2\).

\begin{lemma}\label{lem:T2factorization}
  The map
  \[K_0\nb{\nb{\mathbb{S}^0\oplus\mathbb{S}^0}\otimes\nb{\mathbb{S}^0\oplus\mathbb{S}^{-2}\oplus\mathbb{S}^{-4}}\otimes DA(1)}\to K_0 E\]
  factors, in \(K_0E\)-comodules, through the injective map \(K_0E^{hF_{3/2}}\to K_0 E\).
\end{lemma}
\begin{proof}
  Since \(K_0E^{hF_{3/2}}\to K_0 E\) is injective we only need to show that it factors as a map of sets.

  First note that, for dimensional and connectivity reasons, the map \(DA(1)\to BP\) factors through \(T(2)\to BP\), where \(T(2)\) is the \(2\)-local summand of Ravenel's \(X(4)\), see \cite[thm.~3.1.5]{Devalapurkar_higherThom}.
  The inclusion map identifies \(K_*T(2)\) with the subset \[K_*[t_1,t_2]\subset K_*[t_1,t_2,\ldots]=K_*BP.\]
  So the image of the map in question is included in the sub-\(\ff_4\)-algebra generated by
  \[\eta_R(\zeta),\eta_R(u)/\eta_L(u), ut_1,\text{ and }u^3t_2.\]
  Via the isomorphism \(K_0E=C^0\nb{\mathbb{G},\ff_4}\) we can identify these with the following functions:
  Let the element \(\gamma\in\mathbb{S}\) be presented as a power series
  \[\gamma=\sum_{i\geq0}a_iT^i\]
  with \(a_i\in\{0,1,\zeta,\zeta^2\}\) and \(a_0\neq 0\), and let \(\epsilon\in\{0,1\}\).
  Then we have, by Proposition \ref{prop:hopfstruct} and \cite[thm.~6.2.2]{k2localmoore},
  \[\eta_R(\zeta)(\gamma\sigma^{\epsilon})=\sigma^{\epsilon}(\zeta)\text{ and }\nb{\eta_R(u)/\eta_L(u)}(\gamma\sigma^{\epsilon})=a_0.\]

  To get formulas for \(ut_1\) and \(u^3t_2\) we will use the calculation of the action of \(\mathbb{S}\) on \(E_*\) of \cite[§~6]{k2localmoore}, though we will denote by \(f_i(\gamma)\) what she calls \(t_i(\gamma)\) to avoid confusion with the elements \(t_i\in BP_*BP\).

  Recall that in \(BP_*BP\) we have the relations
  \begin{align*}
    2t_1&=\eta_R(v_1)-v_1\\
    2t_2&=\eta_R(v_2)-v_2+3v_1^2t_1+5v_1t_1^2+4t_1^3
  \end{align*}
  where we use Hazewinkel generators.
  In \(E_*=\mathbb{W}(\mathbb{F}_4)[[u_1]][u^{\pm}]\) we have that \(uv_1=3u_1\) and \(u^3v_2=u_1^3-1\).
  Mapping these relations through
  \[BP_*BP\to E^\vee_*E=C^0(\mathbb{G},E_*)\]
  we thus find that
  \begin{align*}
    2t_1(\gamma\sigma^\epsilon)&=3\frac{\gamma.u_1}{\gamma.u}-3\frac{u_1}{u}=\frac{2}{u}f_1(\gamma)/f_0(\gamma)^2\\
    \implies ut_1(\gamma\sigma^\epsilon)&=f_1(\gamma)/f_0(\gamma)^2.
  \end{align*}

For \(u^3t_2\) we have to work a bit harder:
  \begin{align*}
    2u^3t_2&=u^3\frac{(\gamma.u_1)^3-1}{(\gamma.u)^3}-u_1^3+1+27u_1^2f_1/f_0^2+15u_1f_1^2/f_0^4+4f_1^3/f_0^6\\
    &=1-1/f_0^3+29u_1^2f_1/f_0^2+(4/3+15)u_1f_1^2/f_0^4+(8/27+4)f_1^3/f_0^6\\
    &=1-1/f_0^3-u_1^2f_1/f_0^2-u_1f_1^2/f_0^4+(4)\\
    &=(f_0^4-f_0-u_1^2f_1f_0^2-u_1f_1^2)/f_0^4+(4)\\
    %\star&=(f_0+2f_3+2u_1f_0f_2+u_1^2f_1f_0^2+u_1f_1^2-f_0-u_1^2f_1f_0^2-u_1f_1^2)/f_0^4+(4)\\
    &\overset{\star}{=}(f_0-2f_3-3u_1f_1^2-2u_1f_0f_2-3u_1^2f_0^2f_1-f_0-u_1^2f_1f_0^2-u_1f_1^2)/f_0^4+(4)\\
    &=2f_3/f_0^4+(4,2u_1)
  \end{align*}
  where we suppressed the evaluation at \(\gamma\) and in the step marked with \(\star\) used the formula \cite[prop.~6.3.3]{k2localmoore}:
  \[f_0^4\equiv f_0-2f_3-3u_1f_1^2-2u_1f_0f_2-3u_1^2f_0^2f_1+(4).\]
  Now we can divide by \(2\):
  \[u^3t_2(\gamma\sigma^\epsilon)=f_3(\gamma)/f_0(\gamma)^4+(2,u_1).\]
  By a small variation of \cite[cor.~6.3.5]{k2localmoore} we now find that, in \(K_0E\), we have
  \[(ut_1)(\gamma\sigma^\epsilon)=a_1/(a_0)^2,(u^3t_2)(\gamma\sigma^\epsilon)=a_2/a_0.\]
  Since all of the functions only depend on \(\epsilon, a_0, a_1, a_2\) they factor over \(\mathbb{G}/F_{3/2}\), so they already lie in the image of \(K_0E^{hF_{3/2}}\).
\end{proof}

\begin{corollary}\label{cor:sectionTauBP}
  The \(\Sigma'\)-comodules \(K_0\nb{E^{hF_{3/2}}\otimes MU\langle  6\rangle}\) and \(K_0\nb{E^{hF_{3/2}}\otimes MString}\) are splittable.
\end{corollary}
\begin{proof}
	We apply Proposition \ref{prop:splittableCriterion} using the section and factorization we have just constructed.
\end{proof}

\section{The structure of \(K_0MString\) and \(E_0^\vee MString\)}\label{section:comodstruct}
By the results of \cite{KitchlooLauresWilson_BO} we know that \(K_*MString\), \(K_*BString\), and \(K_*BSpin\) are concentrated in even degrees, thus are comodules over \(\Sigma\).
In this section we will use the fiber sequence
\[K(\zz,3)\to BString \to BSpin\]
and the Thom isomorphism to investigate the comodule structure of \(K_0MString\).
We will also investigate the relations between Morava \(E\)-cohomology and completed homology of \(BString\) and \(BSpin\).

Let us fix some notation.
We have the following commutative diagram of infinite loop spaces:
% https://q.uiver.app/#q=WzAsNixbMCwwLCJLKFxcbWF0aGJie1p9LDMpIl0sWzEsMCwiSyhcXG1hdGhiYntafSwzKSJdLFsyLDAsIksoXFxtYXRoYmJ7Wn0sMykiXSxbMCwxLCJCVVxcbGFuZ2xlIDZcXHJhbmdsZSJdLFsxLDEsIkJTdHJpbmciXSxbMiwxLCJCVVxcbGFuZ2xlIDZcXHJhbmdsZSJdLFswLDEsIlxcaWQiXSxbMSwyLCJcXGNkb3QyIl0sWzMsNCwicmUiXSxbNCw1LCJjIl0sWzAsMywiaiIsMl0sWzEsNCwiaSIsMl0sWzIsNSwiaiIsMl1d
\[\begin{tikzcd}
	{K(\mathbb{Z},3)} & {K(\mathbb{Z},3)} & {K(\mathbb{Z},3)} \\
	{BU\langle 6\rangle} & BString & {BU\langle 6\rangle}
	\arrow["\id", from=1-1, to=1-2]
	\arrow["\cdot2", from=1-2, to=1-3]
	\arrow["\mathrm{re}", from=2-1, to=2-2]
	\arrow["c", from=2-2, to=2-3]
	\arrow["j"', from=1-1, to=2-1]
	\arrow["i"', from=1-2, to=2-2]
	\arrow["j"', from=1-3, to=2-3]
\end{tikzcd}\]
Let \(\tau_W\in E^0MString\) be the Ando-Hopkins-Rezk orientation, denote by \(\tau_\sigma\) its restriction to \(MU\langle6\rangle\) along \(\mathrm{re}\), called the sigma orientation, and let \(\tau_U\) be the complex orientation associated to the coordinate \(-x/y\) on the Weierstrass curve of \(E\).
All of these admit structures of \(E_\infty\)-ring maps (for the complex orientation this follows from \cite[Theorem~6.5.3]{AlgPow}).
Following \cite{cannibal} let \(r_U\in E^0BU\langle6\rangle\) be the unique class which corresponds to \(\tau_\sigma\) under the Thom isomorphism
\[E^0MU\langle6\rangle\xrightarrow{\sim}E^0BU\langle6\rangle\]
using \(\tau_U\), let \(r=c^*r_U\) be its restriction to \(BString\) via the complexification map, and let \(r_K=j^*r_U\) be its restriction to \(K(\zz,3)\).
We then get isomorphisms
\[E^*BSpin[\![r-1]\!]\xrightarrow{\sim}E^*BString\text{ and }E^*[\![r_K-1]\!]\xrightarrow{\sim}E^*K(\zz,3).\]

\begin{construction}\label{cons:sandres}
  Consider the composite
  \[E^0BSpin\xrightarrow{pr^*}E^0BString\to E^0BString/(r-1).\]
  It is a bijection by the above discussion, and a continuous map between compact Hausdorff spaces if we give \(E^0BSpin\) and \(E^0BString\) the natural topology of Section 11 in \cite{666}, thus a homeomorphism.
  Composing the projection with the inverse of the composite we obtain a continuous retraction of \(pr^*\) we will call 'evaluation at \(r=1\)' and denote by
  \[x\mapsto x\vert_{r=1}.\]

  \Cref{lem:contdual2} further furnishes a section of \(pr_*\) we will denote by
  \[s:E_0^\vee BSpin\to E_0^\vee BString.\]
  Modding out \(\mm\) we also obtain analogous maps between the \(K\)-cohomologies and homologies.
\end{construction}

\begin{lemma}\label{lem:prodinr}
  The maps \(T_k:E^0BString\to E^0BSpin\) sending a power series in \((r-1)\) to the coefficient of \((r-1)^k\) are continuous with respect to the natural topology.
  In particular, they assemble to a homeomorphism
  \[E^0BString\xrightarrow{\sim}\prod_{k\geq0}E^0BSpin.\]
  The analogous statement for \(E^0K(\zz,3)\) and the coefficients of \((r_K-1)^k\) holds as well.
\end{lemma}
\begin{proof}
  We focus on the first statement, the proof for \(E^0K(\zz,3)\) being much the same.
  We have already constructed \(T_0=(x\mapsto x\vert_{r=1})\) as a continuous map in \cref{cons:sandres}.
  Consider the ideal \((r-1)E^0BString\subset E^0BString\).
  As the image of multiplication by \(r-1\) it is closed, and thus compact Hausdorff.
  Since \(E^0BString\) has no \((r-1)\)-torsion we get that
  \[E^0BString\to (r-1)E^0BString,\,x\mapsto (r-1)x\]
  is a continuous bijection between compact Hausdorff spaces and thus a homeomorphism.
  Denote its inverse by \(f\).
  We then have the recursive formula
  \[T_{k+1}(x)=T_0\nb{f^{k+1}\nb{x-(r-1)^kT_k(x)-\ldots-T_0(x)}}\]
  writing every \(T_k\) as a composite of continuous functions, showing that they are continuous themselves.

  Finally, the map from \(E^0BString\) to the countable product of \(E^0BSpin\) induced by the \(T_k\) is a continuous bijection between compact Hausdorff spaces.
\end{proof}

\begin{proposition}\label{prop:EBstringfunctions}
  In completed Morava \(E\)-homology we have an isomorphism
  \[E_0^\vee BString\xrightarrow{\sim}C^0(\zz_2,E_0^\vee BSpin),\,x\mapsto(n\mapsto pr_*(x\smallfrown r^n))\]
  where we give \(E_0^\vee BSpin\) the \(\mm\)-adic topology.
\end{proposition}
\begin{proof}
  Our strategy will be to define a map on a pro-basis of the source which, by construction, is an isomorphism.
  A direct computation will then show that it agrees with the above map.

  Let us first show that the target is an \(L\)-complete and, in fact, pro-free \(E_0\)-module.
  Let \(\{y_i\}_{i\in I}\) be a pro-basis of \(E_0^\vee BSpin\), so that we have an isomorphism
  \[\nb{\bigoplus_{i\in I}E_0}_\mm^\wedge\xrightarrow{\sim}E_0^\vee BSpin.\]
  Also note that \(C^0(\zz_2,E_0)\) is pro-free with pro-basis given by the functions
  \[\phi_k:\zz_2\to E_0,\,n\mapsto{n\choose k},\]
  for \(k\geq0\), by a result of Mahler, see \cite[Theorem~11]{CahenChabert_intPoly}.
  By Lemma \ref{lem:compactCompleteQuotients} we also have that \(C^0(\zz_2,E_0)/\mm^n=C^0(\zz_2,E_0/\mm^n)\).
  We find that
  \begin{align*}
    C^0(\zz_2,E_0^\vee BSpin)&=C^0\nb{\zz_2,\lim_n\bigoplus_{i\in I}E_0/\mm^n}=\lim_nC^0\nb{\zz_2,\bigoplus_{i\in I}E_0/m^n}\\
    &=\lim_n\bigoplus_{i\in I}C^0\nb{\zz_2,E_0/m^n}=\lim_n\bigoplus_{i\in I}\bigoplus_{k\geq 0}E_0/\mm^n\\
    &=\nb{\bigoplus_{i\in I,k\geq 0}E_0}_\mm^\wedge
  \end{align*}
  where we used that \(\zz_2\) is a compact space.
  So we see that \(C^0\nb{\zz_2,E_0^\vee BSpin}\) is pro-free with a pro-basis given by \(n\mapsto \phi_k(n)y_i\).

  Using \cref{lem:contdual2}, \cref{rem:contdual2}, and \cref{lem:prodinr} we define a pro-basis
  \[\{\alpha_{k,i}\}_{k\geq0,i\in I}\]
  of \(E_0^\vee BString\) such that, for all \(c\in E^0BString\), we have
  \[\langle \alpha_{k,i},c\rangle =\langle \alpha_{k,i},\sum_{n\geq0}(r-1)^npr^*c_n\rangle=\langle y_i, c_k\rangle.\]
  For any \(n\in\zz_2\) and \(T\in E^0BSpin\) we have
  \begin{align*}
    \langle pr_*\nb{\alpha_{k,i}\smallfrown r^n},T\rangle&=\langle \alpha_{k,i},r^npr^*T\rangle=\left\langle y_i,{n\choose k} T\right\rangle
  \end{align*}
  so that
  \[pr_*\nb{\alpha_{k,i}\smallfrown r^n}=\phi_k(n)y_i.\]

  The map
  \[E_0^\vee BString\to C^0\nb{\zz_2,E_0^\vee BSpin},\,\alpha_{k,i}\mapsto \phi_ky_i.\]
  is an isomorphism as it sends a pro-basis to a pro-basis, and it agrees on this pro-basis (and thus on all elements by \(\mm\)-adic convergence) with the map in question.
\end{proof}
\begin{remark}
  One can directly show that for every \(x\in E_0^\vee BString\) the map
  \[\zz_2\to E_0^\vee BSpin,\,n\mapsto pr_*(x\smallfrown r^n)\]
  is continuous if we give the target the natural topology.
  Using the proposition one can then show that it does not matter whether we use the natural or the \(\mm\)-adic topology in the proposition.
\end{remark}
\begin{remark}
  Using that the integers are dense in \(\zz_2\) and that \(r\) is a ring map one can show that the map we just constructed is an isomorphism of rings if we give \(C^0\nb{\zz_2,E_0^\vee BSpin}\) the pointwise ring structure.
  We also see that under this isomorphism \(pr_*\) corresponds to evaluation at \(0\) and \(s\) corresponds to the inclusion of the constant functions, thus \(s\) is a ring map.
\end{remark}

The Ando-Hopkins-Rezk orientation yields the Thom isomorphism
\[t_*:K_0MString\to K_0BString.\]
This is, very importantly, not a morphism of \(\Sigma\)-comodules, since it does not commute with the action of \(\mathbb{G}\).
The difference is encoded in the so-called cannibalistic classes:
\begin{definition}
  Let \(g\in\mathbb{G}\) and consider the Ando-Hopkins-Rezk orientation \(\tau_W\in E^0MString\).
  By the cohomological Thom isomorphism there is a unique class
  \[\theta^g\in E^0BString\]
  such that \(g.\tau_W=\tau_W\theta^g\).
  Furthermore, denote by \(\chi^g\in K^0BSpin\) its reduction mod \(\mm\) and 'evaluation at \(r=1\)'.
\end{definition}
It is in general not easy at all to determine the \(\theta^g\) in terms of more geometrically defined classes.
See \cite{cannibal} for an extensive discussion of these classes and their properties.
For now, we will only need the following result.
\begin{lemma}
  In \(K^0BString\) we have the equality
  \[\theta^g=r^{\frac{|g|-1}{2}}pr^*\chi^g.\]
\end{lemma}
\begin{proof}
  Let \(\{y_i\}_{i\in I}\) be a pro-basis of \(E_0^\vee BSpin\) and let \(\alpha_j\in E_0^\vee K(\zz,3)\) be the elements which correspond, under \cref{rem:contdual2}, to projecting onto the coefficient of \((r_K-1)^j\).
  Then \(\{i_*\alpha_{2k}sy_i\}_{k\geq0,i\in I}\) is a pro-basis of \(E_0^\vee BString\):
  under the isomorphism of \cref{prop:EBstringfunctions} they send \(n\in\zz\) to
  \[pr_*\nb{(i_*\alpha_{2k}sy_i)\smallfrown r^n}=pr_*(i_*(\alpha_{2k}\smallfrown r_K^{2n}))pr_*(sy_i\smallfrown r^n)={2n \choose 2k}y_i,\]
  thus they agree, mod \(\mm\), with the elements \(\alpha_{k,i}\) considered in the proof of \cref{prop:EBstringfunctions}.

  Using that \(r\), \(\theta^g\), and \(\chi^g\) are group-like elements and that \(i^*\theta^g=r_K^{|g|-1}\) we find
  \begin{align*}
    \langle i_*\alpha_{2k}sy_i, \theta^g\rangle&=\langle i_*\alpha_{2k},\theta^g\rangle\langle sy_i, \theta^g\rangle=\langle \alpha_{2k}, i^*\theta^g\rangle\langle y_i,\chi^g\rangle\\
    &=\langle \alpha_{2k},r_K^{|g|-1}\rangle\langle y_i,\nb{pr^*\chi^g}\vert_{r=1}\rangle\\
    &=\langle i_*\alpha_{2k},r^{\frac{|g|-1}{2}}\rangle\langle sy_i,pr^*\chi^g\rangle=\langle i_*\alpha_{2k}sy_i,r^{\frac{|g|-1}{2}}pr^*\chi^g\rangle
  \end{align*}
  and conclude by duality.
\end{proof}
\begin{definition}
  For \(g\in\mathbb{G}\) consider the element \(g.r_u/r_u^{|g|}\in E^0BU\langle 6\rangle\).
  By \cite[Theorem~3.8]{cannibal} it admits a unique lift to \(E^0BSU\) we will denote by \(q_0^g\).
\end{definition}
\begin{proposition}
  Under the isomorphism
  \[K_0MString\xrightarrow{\sim}C^0\nb{\zz_2^\times,K_0BSpin},\,m\mapsto u\mapsto pr_*\nb{(t_*x)\smallfrown r^{(u-1)/2}}\]
  the Morava stabilizer group acts by
  \[(g.m)(u)=g.\nb{m(u/|g|)\smallfrown\nb{\chi^{g^{-1}}\nb{c^*q_0^{g^{-1}}}^{(u-1)/2}}}.\]
\end{proposition}
\begin{proof}
  This follows from the map \(\zz_2^\times\to\zz_2,\,u\mapsto (u-1)/2\) being a homeomorphism, \cref{prop:EBstringfunctions}, and direct computation using the duality of the homological and cohomological Thom isomorphisms.
\end{proof}
In light of this result we will think of \(K_0MString\) as the space of continuous sections of a \(\mathbb{G}\)-equivariant vector bundle on \(\zz_2^\times\) with constant fiber \(K_0BSpin\), where \(\mathbb{G}\) acts on the base \(\zz_2^\times\) by multiplication with the determinant.

Recall that the element \(\pi=1+2\zeta\in\mathbb{G}\) has determinant 3.
\begin{proposition}\label{prop:identifyingfixedpoints}
  Let \(M=K_0MString\) and \(N=K_0BSpin\).
  Then the map
  \[M\to N\oplus N,\,m\mapsto\nb{m(1),(\pi.m)(1)}\]
  restricts to an equivalence
  \[M^{F_{3/2}}\cong N_1\oplus N_1\]
  where \(N_1=\{m(1)\mid m\in M^{F_{3/2}}\}\).
\end{proposition}
\begin{proof}
  First of all, \((\pi.m)(1)\) sits in \(N_1\) as well, since \(\pi\) normalizes \(F_{3/2}\).

  For injectivity note that
  \[(\pi.m)(1)=\pi.\nb{m(1/3)\smallfrown \chi^{\pi^{-1}}}\]
  so we may extract the value of \(m\) at \(1\) and at \(1/3\).
  Since 
  \[\zz_2^\times=(1+4\zz_2)\times\{\pm1\},\] 
  the subgroup \(F_{3/2}\subset\mathbb{G}\) acts transitively on \((1+4\zz_2)\), and 
  \[1\in (1+4\zz_2)\times\{+1\},\, 1/3\in (1+4\zz_2)\times\{-1\}\]
  we can use the \(F_{3/2}\)-invariance of \(m\) to reconstruct its value on all of \(\zz_2^\times\).

  For surjectivity consider \((a(1),b(1))\in N_1\oplus N_1\) for some \(a,b\in M^{F_{3/2}}\).
  Define another continuous section \(c\) by
  \[c(u):=\begin{cases}a(u)\text{ for } u\in (1+4\zz_2)\times\{+1\}\\ (\pi^{-1}.b)(u)\text{ for }u \in (1+4\zz_2)\times\{-1\}\end{cases}.\]
  This section satisfies \(c(1)=a(1)\) and \((\pi.c)(1)=b(1)\), and is \(F_{3/2}\)-invariant since the action does not mix the \(+1\) and \(-1\) components of \(\zz_2^\times\).
\end{proof}
It turns out that \(N_1\) coincides with the fixed points of the action of
\[F_{3/2}'=\{g\in F_{3/2}\mid |g|=1\}\]
on \(N\) given by 
\[g\triangleright n:=g.n\smallfrown1/\chi^g.\]
Note that, while the formula makes sense for any \(g\in \mathbb{G}\), it only determines a group action on the elements with determinant \(1\).
\begin{proposition}
  We have that
  \[N_1=\{n\in N\mid \forall g\in F_{3/2}':n=g\triangleright n\}.\]
\end{proposition}
\begin{proof}
  This follows from \(F_{3/2}'\) being the stabilizer of \(1\in\zz_2^\times\), the action described above being the action on \(N\) viewed as the fiber over \(1\), and \(F_{3/2}\) acting transitively on \(1+4\zz_2\).
\end{proof}

\subsection{Real \(k\)-structures}
In this section we introduce the concept of {\it \(k\)-structures} on a formal group law and review some results connecting the complex oriented homology of connective covers of \(BU\) and \(BO\) to them.
This definition is borrowed from \cite{cube}, compare also \cite{realcube} and \cite{realKL}.
For a good overview, which also includes the relevant results of the unpublished \cite{realcube}, see \cite[Chapter~5]{formalGeometry}.

\begin{definition}
  Let \(F\) be a formal group law on a ring \(R\).
  A {\it \(k\)-structure} on an \(R\)-algebra \(S\) is a power series \(f\in S[\![x_1,\ldots,x_k]\!]\) such that
  \begin{itemize}
    \item \(f(0,x_2,\ldots,x_k)=1\)
    \item \(f(x_{\sigma(1)},\ldots,x_{\sigma(k)})=f(x_1,\ldots,x_k)\) for all \(\sigma\in\Sigma_k\)
    \item for \(k\geq2\):
    \[\frac{f(x_1,x_2,\ldots,x_k)}{f(x_0,x_1,x_3,\ldots,x_k)}=\frac{f(x_0+_Fx_1,x_2,\ldots,x_k)}{f(x_0,x_1+_Fx_2,\ldots,x_k)}.\]
  \end{itemize}

  A real \(1\)-structure is a \(1\)-structure such that \(f(x)=f([-1](x))\), a real \(2\)-structure is a \(2\)-structure such that \(f(x,y)=f([-1](x),[-1](y))\).

  Given a \(k\)-structure \(f\), we can produce a \((k+1)\)-structure 
  \[\delta f(x_1,\ldots,x_{k+1})=\frac{f(x_1,x_3,\ldots,x_{k+1})f(x_2,x_3,\ldots,x_{k+1})}{f(x_1+_Fx_2,x_3,\ldots,x_{k+1})}.\]
  If \(f\) was a real \(1\)-structure, then \(\delta f\) will be a real \(2\)-structure.

  All of these are corepresentable by \(R\)-algebras \(C_kF,\,C_1^rF,\,C_2^rF\) (and maps between them) which are natural in both \(F\) and \(R\).
\end{definition}

\begin{theorem}[{\cite[Theorem~2.31]{cube}}]
  Let \(E\) be an even periodic complex oriented homotopy commutative ring spectrum with associated formal group law \(F\) over \(E_0\).
  Then there are, for \(k\leq 3\), natural isomorphisms
  \[C_kF\to E_0BU\langle2k\rangle\]
  classifying the power series associated to the bundle
  \[[L_1-1]\otimes\ldots\otimes[L_k-1]:\nb{\cc P^\infty}^k\to BU\langle2k\rangle.\]
  The maps corepresenting \(\delta\) correspond to the covering maps
  \[BU\langle2k+2\rangle\to BU\langle2k\rangle.\]
\end{theorem}
\begin{theorem}[{\cite[Theorem~1.6]{realcube}, see also \cite[Theorem~5.6.16]{formalGeometry}}]
  For E an even periodic Morava \(K\)-theory (of height \(1\) or \(2\) if \(p=2\)) there is a natural isomorphism
  \[C_2^rF\to E_0BSpin.\]
\end{theorem}
\begin{theorem}[{\cite[Theorem~1.1]{realKL}}]\label{thm:RealHondaIso}
  For \(F\) the Honda formal group law over \(\mathbb{F}_2\) (or if \(2\) is invertible in the coefficients) the map corepresenting \(\delta\) gives an isomorphism
  \[C_2^rF\to C_1^rF.\]
\end{theorem}

By definition a \(1\)-structure is a power series \(f(x)\) such that \(f(0)=1\).
Thus, \(C_1F\cong R[b_1,b_2,\ldots]\) and the universal example is \(f(x)=1+\sum_{k\geq1}b_kx^k\).
Note that the choice of the \(b_i\) is natural in \(F\) and \(R\).
\(C^r_1F\) is then a quotient of this polynomial algebra, but in many cases it is itself polynomial.
\begin{proposition}[{Compare \cite[Propositions~2.5, 2.7]{realKL}}]
  Let \(E\) be a complex oriented Morava \(E\)-theory of finite height with quotient field \(k\) of characteristic~\(2\), and let \(F\) be the associated formal group law.
  Then
  \[\nb{C_1^rF}_{\mm}^\wedge \cong \nb{E_0[b_2,b_4,b_6,\ldots]}_\mm^\wedge.\]
\end{proposition}
\begin{proof}
  As in the proof of Proposition 2.5 in \cite{realKL} we see that
  \[\nb{E_0[b_2,b_4,b_6,\ldots]}/\mm^n\to C_1^rF/\mm^n\]
  is surjective for \(n=1\), and then for \(n\geq1\) by Nakayama's lemma.

  Inspecting the proof of Proposition 2.7 in \cite{realKL} we are reduced to proving that the classes
  \[(-1)^k\nb{b_k^2+2\sum_{j=1}^kb_{k+j}b_{k-j}}+\text{ terms of degree less than } 2k\]
  are algebraically independent in \(E_0/\mm^n[b_1,b_2,\ldots]\), where we give \(b_i\) degree \(i\).

  The \(b_k^2\) parts are clearly algebraically independent, and adding multiples of \(2\) to them doesn't change that since \(2\) is nilpotent in \(E_0/\mm^n\).
  Adding the terms of lower degree also doesn't change the algebraic independence by looking at the top degree part.

  We saw that the map is an isomorphism after quotienting out finite powers of \(\mm\), thus it is an isomorphism after completion.
\end{proof}

\begin{remark}
  The above also suggests a way to determine how to write \(b_{2k+1}\) in terms of \(b_2,\ldots,b_{2k}\):
  since there is no \(2\)-torsion in \(\nb{C_1^rF}_{\mm}^\wedge\) it is enough to look at the coefficients of \(x^{2k+1}\) of \(2f(x)=f(x)+f([-1](x))\).
  This expresses \(2b_{2k+1}\) as a linear combination of \(b_1,\ldots,b_{2k}\).
\end{remark}
\begin{proposition}
  Let \(E\) be a complex oriented Morava \(E\)-theory of finite height with quotient field \(k\) of characteristic \(2\), and let \(F\) be the associated formal group law.
  Then the map corepresenting \(\delta\) gives an isomorphism
  \[\nb{C_2^rF}_\mm^\wedge \to \nb{C_1^rF}_\mm^\wedge.\]
\end{proposition}
\begin{proof}
  In the algebraic closure of \(k\) the formal group law \(F\) is ismorphic to the Honda formal group law.
  So, by naturality in the formal group law, \cref{thm:RealHondaIso}, and that field extensions are faithfully flat, we find that
  \[\nb{C_2^rF}/\mm\to\nb{C_1^rF}/\mm\]
  is an isomorphism.
  By Nakayama's lemma we find that
  \[\nb{C_2^rF}/\mm^n\to\nb{C_1^rF}/\mm^n\]
  is surjective for all \(n\geq1\).

  Denote the kernel of the above map by \(K_n\).
  We saw that \(K_1=0\).
  In the previous proposition we have also seen that \(\nb{C_1^rF}/\mm^n\) is a free \(E_0/\mm^n\) module, so in particular it is flat.
  A diagram chase now shows that \(K_n=K_{n+1}/\mm^n\).
  Since \(\mm^n\) is a nilpotent ideal in \(E_0/\mm^{n+1}\) we can use Nakayama's lemma again to show that all \(K_n\) vanish.
\end{proof}

The situation can be summarized using the following diagram:
% https://q.uiver.app/#q=WzAsOCxbMSwxLCJFXzBeXFx2ZWUgQlNVIl0sWzMsMSwiRV8wXlxcdmVlIEJTcGluIl0sWzQsMywiRV8wXlxcdmVlIEJTTyJdLFsyLDIsIihDXzFeckYpX1xcbW1eXFx3ZWRnZSJdLFsyLDAsIihDXzJeckYpX1xcbW1eXFx3ZWRnZSJdLFswLDAsIihDXzJGKV9cXG1tXlxcd2VkZ2UiXSxbMSwzLCJFXzBeXFx2ZWUgQlUiXSxbMCwyLCIoQ18xRilfXFxtbV5cXHdlZGdlIl0sWzEsMiwicHJfKiIsMV0sWzMsMiwiIiwyLHsic3R5bGUiOnsidGFpbCI6eyJuYW1lIjoiaG9vayIsInNpZGUiOiJ0b3AifX19XSxbNCwxLCJcXHNpbWVxIiwxXSxbNSwwLCJcXHNpbWVxIiwxXSxbNyw2LCJcXHNpbSIsMV0sWzAsNl0sWzYsMiwicmVfKiIsMV0sWzAsMSwiIiwxLHsic3R5bGUiOnsiaGVhZCI6eyJuYW1lIjoiZXBpIn19fV0sWzUsN10sWzQsM10sWzUsNF0sWzcsMywiIiwxLHsic3R5bGUiOnsiaGVhZCI6eyJuYW1lIjoiZXBpIn19fV0sWzYsMywiIiwwLHsic3R5bGUiOnsiYm9keSI6eyJuYW1lIjoiZGFzaGVkIn0sImhlYWQiOnsibmFtZSI6ImVwaSJ9fX1dLFsxLDMsIlxcc2ltZXEiLDEseyJzdHlsZSI6eyJib2R5Ijp7Im5hbWUiOiJkYXNoZWQifX19XV0=
\[\begin{tikzcd}
	{(C_2F)_\mm^\wedge} && {(C_2^rF)_\mm^\wedge} \\
	& {E_0^\vee BSU} && {E_0^\vee BSpin} \\
	{(C_1F)_\mm^\wedge} && {(C_1^rF)_\mm^\wedge} \\
	& {E_0^\vee BU} &&& {E_0^\vee BSO}
	\arrow[from=1-1, to=1-3]
	\arrow["\simeq"{description}, from=1-1, to=2-2]
	\arrow[from=1-1, to=3-1]
	\arrow["\simeq"{description}, from=1-3, to=2-4]
	\arrow[from=1-3, to=3-3]
	\arrow[two heads, from=2-2, to=2-4]
	\arrow[from=2-2, to=4-2]
	\arrow["\simeq"{description}, dashed, from=2-4, to=3-3]
	\arrow["{pr_*}"{description}, from=2-4, to=4-5]
	\arrow[two heads, from=3-1, to=3-3]
	\arrow["\simeq"{description}, from=3-1, to=4-2]
	\arrow[hook, from=3-3, to=4-5]
	\arrow[dashed, two heads, from=4-2, to=3-3]
	\arrow["{\mathrm{re}_*}"{description}, from=4-2, to=4-5]
\end{tikzcd}\]
In particular, there are unique classes \(\widetilde{b_{2i}}\in E_0^\vee BSpin\) such that \(pr_*\widetilde{b_{2i}}=\mathrm{re}_*b_{2i}\), and \(E_0^\vee BSpin\) is a completed polynomial algebra on these classes.
Notice that the map \((C_1^rF)_\mm^\wedge\to E_0^\vee BSO\) is also injective mod \(\mm\), which follows directly from the proof of Proposition 2.7 in \cite{realKL}.
\begin{proposition}\label{prop:geometricgeneratorsKBSpin}
  Consider the map
  \[S^1\xrightarrow{z\mapsto 1/z}S^1=Spin(2)\longrightarrow Spin\]
  and denote by \(e_i\in E_0^\vee BSpin\) the image of \(\beta_i\in E_0^\vee BS^1\) under it.
  Then
  \[pr_*e_i\equiv \begin{cases}\mathrm{re}_*b_{i/4}+\mm & \text {if }4\vert i\\
   0+\mm&\text{else.}\end{cases}\]
  In particular,
  \[E_0^\vee BSpin=\nb{E_0[e_8,e_{16},e_{24},\ldots]}_\mm^\wedge.\]
  
\end{proposition}
\begin{proof}
  Denote the map \(S^1\to S^1,\,z\mapsto z^{-2}\) by \(f\).
  Since the projection
  \[Spin\longrightarrow SO\] 
  induces the squaring map
  \[S^1=Spin(2)\longrightarrow SO(2)=S^1,\,z\mapsto z^2\]
  we find that \(pr_*e_i=\mathrm{re}_*f_*\beta_i\).
  Up to \(\mm\) we know that the \((-2)\)-series of the formal group law is given by \([-2](x)=x^4\).
  We calculate
  \begin{align*}
    f_*\beta_i&=\sum_{j\geq0}\beta_j\langle f_*\beta_i,z^j\rangle=\sum_{j\geq0}\beta_j\langle \beta_i,[-2](z)^j\rangle\\
    &=\sum_{j\geq 0}\beta_j\langle\beta_i,z^{4j}\rangle=\sum_{j\geq0}\beta_j\delta_{i}^{4j}
  \end{align*}
  which shows the first claim.
  The second now follows since \(e_{8i}=\widetilde{b_{2i}}+\mm\).
\end{proof}

\subsection{Calculating and detecting \(N_1\)}
Let us try to determine some classes which lie in \(N_1\), together with cohomology classes which detect them.
We will see that these cohomology classes will be involved in determining some of the projections \(Q_j\) of \cref{maintheorem}.

Recall that the action of \(F_{3/2}'\) we are interested in is the one twisted by the elements \(\chi^g\in K^0BSpin\) such that
\[\theta^g\equiv r^{\frac{|g|-1}{2}}pr^*\chi^g+\mm.\]
To determine the invariants we will need to calculate \(\chi^g\) for the generators listed in \cref{lem:generatorsF32prime}.
Since \(\chi^g\) is a group-like element pairing with it induces an \(\ff_4\)-algebra map \(K_0BSpin\to K_0\).
As we have seen in the previous subsection, for our Morava \(K\)-theory of height 2, such algebra morphisms are determined by real 1-structures.

\begin{proposition}\label{prop:real1Struct}
  Let \(g\in\mathbb{S}\).
  There is a unique real 1-structure \(l_g(z)\) such that, as 3-structures, we have
  \begin{align*}
    \delta(\mathrm{re}^*\chi^g)(x,y,z)&=\frac{s(g(x),g(y),g(z))\,\delta^2\!\nb{\frac{g}{g'(0)}}\!(x,y,z)}{s(x,y,z)^{\frac{|g|+1}{2}}s([-1](x),[-1](y),[-1](z))^{\frac{|g|-1}{2}}}\\
    &\\
    &=\delta^2l_g(x,y,z)
  \end{align*}
  where \(s\) is the standard 3-structure on the formal group law associated to the Weierstrass curve \(y^2+y=x^3\), see \cite[Appendix~B]{cube} or \cite[§~3]{tmf03}.
  In particular, there is a canonical lift of \(\mathrm{re}^*\chi^g\) to a group-like class \(l_g\) in \(K^0BU\).
\end{proposition}
\begin{proof}
  The first equality follows from
  \[\mathrm{re}^*pr^*\chi^g=\frac{g.r_U}{r_U^{\frac{|g|+1}{2}}\overline{r_U}^{\frac{|g|-1}{2}}}\frac{g.\tau_U}{\tau_U}\]
  and identifying the 3-structure associated to the right hand side.
  That this is given by \(\delta^2\) of some real 1-structure comes from the equivalence \(C^2_rF\to C^1_rF\).
  Uniqueness now follows from \(K^0BSpin\to K^0BString\) being injective and \(K^0BString\to K^0BU\langle6\rangle\) being injective on group-like elements.

  Finally, the real \(1\)-structure \(l_g\) determines an algebra morphism \(K_0BU\to K_0\) sending \(b_i\) to the coefficient of \(z^i\) in \(l_g(z)\), which by duality is determined by a class in \(K^0BU\) we shall also denote by \(l_g\).
\end{proof}

Since we know formulas for the \(3\)-structre \(s\), the \(-1\)-series, and the formal group law, we can approximate \(l_g\) for the generators of \(F'_{3/2}\).
By direct computation\footnote{To determine these approximations one needs to calculate the expression of \cref{prop:real1Struct} roughly up to order 32 in \((x,y,z)\), which we did using the SageMath software \cite{sagemath}. We included the code in the arXiv ancillary files.} we find
\begin{align*}
  l_{\alpha^2}(z)&=1+z^6+(z^8)\\
  l_{[i,\alpha]}(z)&=1+z^2+z^4+z^5+z^6+(z^8)\\
  l_{[j,\alpha]}(z)&=1+\zeta z^2+\zeta^2z^4+\zeta z^5+z^6+(z^8).
\end{align*}

We have seen that \(N=K_0BSpin\) is a polynomial \(\ff_4\)-algebra on classes \(\widetilde{b_{2i}}\) such that their projections to \(BSO\) coincide with the realfications of the classes \(b_{2i}\in K_0BU\).
\begin{proposition}
  The action of elements \(g\in \mathbb{S}\) with determinant \(1\) on the \(\widetilde{b_{2i}}\) is determined by the following formula:
  \[pr_*\nb{g^{-1}\triangleright\widetilde{b_{2i}}}=\mathrm{re}_*\nb{\mathrm{coeff}_{z^{2i}}\,l_g(z)l_b(g(z))}.\]
  Here \(l_b\) is the power series
  \[l_b(z)=1+\sum_{i\geq1}b_iz^i.\]
\end{proposition}
\begin{proof}
  First of all note that, when \(|g|=1\), we have an equality \(\chi^gg.\chi^{g^{-1}}=1\).
  This allows us to rewrite the action as
  \[g^{-1}\triangleright n=g^{-1}.\nb{n\smallfrown \chi^{g}}.\]

  Since \(K_0BSU\to K_0BSpin\) is surjective we can pick a class \(x\in K_0BSU\) such that \(\mathrm{re}_*x=\widetilde{b_{2i}}\).
  This class will then, after mapping down to \(K_0BU\), agree with \(b_{2i}\) up to elements in the kernel of \(\mathrm{re}_*:K_0BU\to K_0BSO\).
  We can now calculate
  \begin{align*}
    pr_*\nb{g^{-1}\triangleright\widetilde{b_{2i}}}&=pr_*\nb{g^{-1}.\nb{\mathrm{re}_*x\smallfrown \chi^{g}}}=pr_*\mathrm{re}_*\nb{g^{-1}.\nb{x\smallfrown \mathrm{re}^*\chi^{g}}}\\
    &=\mathrm{re}_*pr_*\nb{g^{-1}.\nb{x\smallfrown pr^*l_{g}}}=\mathrm{re}_*\nb{g^{-1}.\nb{pr_*x\smallfrown l_{g}}}.
  \end{align*}
  Now we observe that both acting with \(g^{-1}\) and capping with \(l_{g}\) preserve the kernel of \(\mathrm{re}_*:K_0BU\to K_0BSO\) so we may replace \(x\) by \(b_{2i}\):
  \[pr_*\nb{g^{-1}\triangleright\widetilde{b_{2i}}}=\mathrm{re}_*\nb{g^{-1}.\nb{b_{2i}\smallfrown l_{g}}}.\]
  It will be useful to consider a generating function. In the following let
  \[f:BS^1=BU(1)\longrightarrow BU\]
  be the usual inclusion map so that \(b_i=f_*\beta_i\).
  \begin{align*}
    \sum_{i\geq 0}t^i\nb{g^{-1}.\nb{b_{i}\smallfrown l_{g}}}&=\sum_{i\geq 0}t^if_*\nb{\sum_{j\geq0}\beta_j\langle g^{-1}.\nb{\beta_i\smallfrown f^*l_g},z^j\rangle}\\
    &=\sum_{i\geq 0}t^if_*\nb{\sum_{j\geq0}\beta_j\langle \beta_i,l_g(z)g(z)^j\rangle}\\
    &=\sum_{j\geq0}b_jl_g(t)g(t)^j=l_g(t)l_b(g(t))
  \end{align*}
  where we used that the action of \(g\) on \(z\in K^0BS^1\) is given by the underlying power series of \(g\in\mathbb{S}\subset\ff_4[\![t]\!]\).
\end{proof}

By the nature of the action of \(F_{3/2}'\) on \(N\) the subrings
\[\ff_4\!\sb{\widetilde{b_{2i}}\mid 1\leq i\leq n}\subset N\]
are invariant subspaces.
With the above computation of the real 1-structures for the generators of \(F_{3/2}'\) we can now determine the action on \(\ff_4\!\sb{\widetilde{b_2},\widetilde{b_4},\widetilde{b_6}}\):
\bgroup
\def\arraystretch{1.5}
\begin{center}
  \begin{tabular}{ c | c | c | c }
    
    \(\triangleright\) & \(\widetilde{b_2}\) & \(\widetilde{b_4}\) & \(\widetilde{b_6}\) \\ \hline
    \(\alpha^{-2}\) & \(\widetilde{b_2}\) & \(\widetilde{b_4}\) & \(\widetilde{b_6}+1\) \\ \hline
    \([i,\alpha]^{-1}\) & \(\widetilde{b_2}+1\) & \(\widetilde{b_4}+\widetilde{b_2}+1\) & \(\widetilde{b_6}+\widetilde{b_4}+\widetilde{b_2}+1\) \\ \hline
    \([j,\alpha]^{-1}\) & \(\widetilde{b_2}+\zeta\) & \(\widetilde{b_4}+\zeta\widetilde{b_2}+\zeta^2\) & \(\widetilde{b_6}+\zeta\widetilde{b_4}+\zeta^2\widetilde{b_2}+1\) \\
  \end{tabular}
\end{center}
\egroup

\begin{proposition}
  Let \(x,y,z\) denote \(\widetilde{b_2},\widetilde{b_4},\widetilde{b_6}\).
  Then the subring of \(F_{3/2}'\)-invariants in \(\ff_4[x,y,z]\) is a polynomial \(\ff_4\)-algebra on the elements
  \begin{align*}
    c=&x^4+x,\\
    d=&y^4+y+x^5+x^2,\text{ and}\\
    e=&z^2+z+y^2x^2+yx+x^6+x^3.
  \end{align*}
\end{proposition}
\begin{proof}
  The claim that these elements are invariant follows from direct computation, and that they are algebraically independent from looking at their leading terms.
  It remains to show that the above elements do generate all invariant elements.

  Let \(p=x^n+p_1x^{n-1}+p_2x^{n-2}+\ldots\) be an invariant polynomial only involving~\(x\).
  We want to show that its degree must be a multiple of 4.
  Acting with \([i,\alpha]\) we get the equation \(p_1=n+p_1\), thus \(n\) must be even, say \(n=2m\).
  Acting with both \([i,\alpha]\) and \([j,\alpha]\) we now find the equations
  \[p_2=m+p_1+p_2=\omega^2m+\omega p_1+p_2\]
  so that also \(m\) must be even.
  Since the degree of any invariant polynomial in \(x\) is a multiple of 4 and \(c=x^4+\ldots\) we can write it as a polynomial in \(c\).

  Let \(p=y^np_0(x)+y^{n-1}p_1(x)+y^{n-1}p_2(x)+\ldots\) be an invariant polynomial in \(x\) and \(y\).
  We want to show that \(p_0\) is invariant and that \(n\) is a multiple of 4.
  Acting with \([i,\alpha]\) and \([j,\alpha]\) and comparing coefficients of \(y\) we find that \(p_0\) must be invariant, and we get the equation
  \[p_1(x)=p_1(x+1)+n(x+1)p_0(x+1).\]
  Using that \(p_0\) is already known to be invariant one finds
  \[n(x+1)p_0(x)=p_1(x)+p_1(x+1)=nxp_0(x)\]
  showing that \(n=2m\) must be even, and that \(p_1\) must be invariant as well.

  Assuming \(m\) is odd and looking at the coefficient of \(y^{n-2}\) a similar trick gives \(p_0=p_1\) and \(p_0=\zeta p_1\), so that \(p_0=0\), a contradiction.
  Thus \(p\) is a polynomial in \(c\) and \(d\).

  Finally, let \(p=z^np_0(y,x)+z^{n-1}p_1(y,x)+\ldots\) be invariant.
  We want to show that \(n\) is even and that \(p_0\) is invariant.
  Acting with \(\alpha^2\) and looking at the coefficient of \(z^{n-1}\) we get the equation
  \[p_1=p_1+np_0\]
  so that \(n\) must be even.
  Invariance of \(p_0\) again follows from looking at leading coefficients after acting with \([i,\alpha]\) and \([j,\alpha]\).
\end{proof}

Now let us discuss how the cohomology of \(BSpin\) and the pairing with the homology behave.
According to \cite{tmf03} there are classes \(p_i\in E^{4i}BSpin\) such that
\begin{itemize}
  \item \(E^*BSpin=E^*[\![p_1,p_2,\ldots]\!]\)
  \item \(\Delta p_k=\sum_{i+j=k}p_i\otimes p_j\)
  \item restricted to each maximal torus in \(Spin(2n)\) we have \[\sum_{k\geq 0}t^kp_k=\prod_{i=1}^{n}1-t\rho^*x_i\bar{x}_i\]
  where \(\rho\) is the map to the maximal torus in \(SO(2n)\), the \(x_i\) are the first Chern classes in \(E^2B(S^1)^{\times n}\), and the \(\bar{x}_i\) their complex conjugates.
\end{itemize}
In particular, restricted to \(Spin(2)=S^1\), we have
\[\sum_{k\geq 0}t^kp_k=1-t[2](x)[-2](x)\]
since \(\rho\) is a double covering.

From \cref{prop:geometricgeneratorsKBSpin} we know that in \(K_0BSpin=\ff_4\!\sb{\widetilde{b_{2}},\widetilde{b_{4}},\ldots}\) we have
\[\widetilde{b_{2i}}=f_*\beta_{8i}\]
where \(f\) is the inversion on \(BS^1\) followed by the inclusion of \(BS^1=BSpin(2)\) into \(BSpin\).
\begin{proposition}
  The only non-zero pairings between the polynomial generators \(\widetilde{b_{2i}}\) and the images of the  Pontryagin classes \(p_j\) under the projection \(E^*BSpin\to K^*BSpin\) (which we will also denote by \(p_j\))  are
  \[\langle \widetilde{b_0},p_0\rangle=\langle\widetilde{b_2},u^{-2}p_1\rangle=1.\]
\end{proposition}
\begin{proof}
  The first pairing is simply the pairing \(\langle1,1\rangle=1\).
  For the second pairing we have
  \begin{align*}
    \langle\widetilde{b_{2i}},u^{-2}p_1\rangle&=\langle \beta_{8i},[2](z)[-2](z)\rangle\\
    &=\langle\beta_{8i},z^4\sum_{k\geq0}\nb{z^4}^{2^k3-2}\rangle\\
    &=\delta_{i,1}
  \end{align*}
  since the only multiple-of-eight power on the right is \(z^8\).
  All other pairings are zero for degree reasons.
\end{proof}

\begin{proposition}\label{prop:detect5classes}
  Pairing with the classes \(1,u^{-8}p_4,u^{-8}i.p_4,u^{-8}j.p_4,u^{-8}k.p_4\) induces an isomorphism
  \[\langle1,cd^3,i.(cd^3),j.(cd^3),k.(cd^3)\rangle \longrightarrow (\ff_4)^5.\]
\end{proposition}
\begin{proof}
  Direct computation shows that the five mentioned classes are mapped to a basis of \((\ff_4)^5\).
\end{proof}

\begin{remark}
  Identifying \(N_1\oplus N_1\) with the fixed points \(\nb{K_0MString}^{F_{3/2}}\) as in \cref{prop:identifyingfixedpoints} we get a residual action of \(\mathbb{G}/F_{3/2}\) on it.
  Since the element \(\pi\in\mathbb{G}\) swaps the two summands and \(-1\in\mathbb{G}\) acts trivially we find that we are really looking at an action of \(G_{48}/\{\pm1\}\) on \(N_1\).
  For an element \(g\in G_{48}\subset\mathbb{G}\) the cannibalistic class \(\theta^g\) equals \(1\) and the determinant equals \(1\), showing that this residual action agrees with the natural action on \(K_0BSpin\).
  That is also why we used \(.\) to denote the action instead of \(\triangleright\) in the above proposition.
\end{remark}

\section{The Milnor-Moore argument}\label{section:milnormoore}
In this section we will use the Milnor-Moore argument together together with the splittability result of \cref{cor:sectionTauBP} to show that \(E^{hF_{3/2}}\otimes MU\langle 6\rangle\) and \(E^{hF_{3/2}}\otimes MString\) split as direct sums of Morava \(E\)-theories.
Then we will use our description of \(K_0MString\) of \cref{section:comodstruct} to prove a first approximation of \cref{maintheorem}.

\begin{theorem}\label{thm:splitMU6}
	There exists a map
	\[D:MU\langle 6\rangle\to L_{K} \bigoplus_{i\in I} E\]
	such that the free \(E^{hF_{3/2}}\)-linearization of \(D\) induces a \(K\)-equivalence
	\[E^{hF_{3/2}}\otimes MU\langle 6\rangle \xrightarrow{\sim} L_{K}\bigoplus_{i\in I} E.\]
  \end{theorem}
  \begin{proof}
	Corollary \ref{cor:sectionTauBP} shows that
	\[K_0\nb{E^{hF_{3/2}}\otimes MU\langle6\rangle}\]
	is a splittable \(\Sigma'\)-comodule.
	By Lemma \ref{lem:identifyP1} the primitive elements are given by
	\[P_1\nb{K_0\nb{E^{hF_{3/2}}\otimes MU\langle6\rangle}}=\nb{K_0\nb{E^{hF_{3/2}}\otimes MU\langle6\rangle}}^{\mathbb{S}}.\]
	Denoting \(K_0MU\langle 6\rangle\) by \(M\) we find that the invariants are given by
	\[C^0\nb{\mathbb{G}/F_{3/2}, M}^\mathbb{S}\xrightarrow{\sim}M^{F_{3/2}}\oplus M^{F_{3/2}},\, f\mapsto (f([e]),f([\sigma])).\]
	Under this equivalence and the one from Lemma \ref{lem:F0=FGP1} the inclusion of the zeroth filtration is given by
	\begin{align*}
	  i:C^0(\ff_4^\times,\ff_4)\otimes \nb{M^{F_{3/2}}\oplus M^{F_{3/2}}}\to C^0(\mathbb{G}/F_{3/2},\ff_4)\otimes M\\
	  s\otimes(m_e,m_\sigma)\mapsto\nb{[\gamma\sigma^\ep]\mapsto s([\gamma])\gamma.(m_{\sigma^\ep})}.
	\end{align*}
	Let \(r_M:M\to M^{F_{3/2}}\) be a retraction and define
	\begin{align*}
	  &r':C^0(\mathbb{G}/F_{3/2},\ff_4)\otimes M\to C^0(\ff_4^\times,\ff_4)\otimes \nb{M^{F_{3/2}}\oplus M^{F_{3/2}}},\\
	  &f\otimes m \mapsto (x\mapsto f([x]))\otimes(r_M(m),0)+(x\mapsto f([x\sigma]))\otimes (0,\sigma.r_M(\sigma.m)).
	\end{align*}
	The composite \(r'\circ i\) is an equivalence
	\[r'(i(s\otimes(a,b)))=\sum_{x\in \ff_4^\times}s(x)\delta_x\otimes(x.a,x.b)\]
	and we define \(r=(r'\circ i)^{-1}\circ r'\).
	Also note that \(\epsilon\otimes\id\circ(r'\circ i)^{-1}=\epsilon\otimes\id\) so that we may use \(r'\) instead of \(r\) in the formula for \(h_r\):
	\[h_r(f\otimes m)=\sum_{\gamma\in\mathbb{S}}\delta_\gamma\otimes\nb{f([\gamma])r_M(\gamma^{-1}.m),f([\gamma\sigma])\sigma.r_M(\sigma\gamma^{-1}.m)}\]
  
	Now we can see how the splitting interacts with the involution:
	\begin{align*}
	  h_r(\sigma.(f\otimes m))&=\sum_{g\in\mathbb{S}}\delta_g\otimes\nb{\sigma\nb{f([\bar{g}\sigma])r_M^\sigma(\bar{g}^{-1}.m)},\sigma\nb{f([\bar{g}])r_M(\bar{g}.m)}}\\
	  &=\sum_{g\in\mathbb{S}}\delta_{\bar{g}}\otimes\nb{\sigma\nb{f([g\sigma])r_M^\sigma(g^{-1}.m)},\sigma\nb{f([g])r_M(g.m)}}
	\end{align*}
	so we must have
	\[\tau(\delta_\gamma\otimes(a,b))=\delta_{\bar{\gamma}}\otimes(\sigma.b,\sigma.a).\]
  
	Let \(\{x_i\}_{i\in I}\) be a basis of \(M^{F_{3/2}}\).
	Then the elements \(y_i=\sigma.x_i\) form a basis, too.
	The map 
	\[\Sigma'\otimes\nb{M^{F_{3/2}}\oplus M^{F_{3/2}}}\to \Sigma\otimes M^{F_{3/2}},\,\delta_\gamma\otimes(ax_i,by_j)\mapsto \delta_\gamma\otimes ax_i + \delta_{\gamma\sigma}\otimes bx_j\]
	is an isomorphism of \(\Sigma'\)-comodules with involution (where we view the target as a direct sum over the basis \(\{x_i\}_{i\in I}\)),  and thus an isomorphism of \(\Sigma\)-comodules by Proposition \ref{prop:twisted=involution}.
  
	Finally, after choosing an isomorphism of \(\Sigma\)-comodules
	\[\Sigma\otimes M^{F_{3/2}}\to K_0\bigoplus_{i\in I}E,\]
	which amounts to a choice of basis of \(M^{F_{3/2}}\), we can lift this isomorphism to the spectrum level as by Lemma \ref{lem:algToSpectral}, yielding a \(K(2)\)-local equivalence
	\[D':E^{hF_{3/2}}\otimes MU\langle 6\rangle\to L_K\bigoplus_{i\in I}E.\]
  
	Let \(\lambda_i:M^{F_{3/2}}\to\ff_4\) be the projections such that
	\[\forall m\in M^{F_{3/2}}:m=\sum\lambda_i(m)x_i\]
	and denote by \(d_i\in K^0MU\langle 6\rangle\) the unique classes such that \(\lambda_i(r_M(m))=\langle m,d_i\rangle\).
	We find that, in total, the isomorphism sends
	\[K_0\nb{E^{hF_{3/2}}\otimes MU\langle6\rangle}\to\Sigma\otimes M^{F_{3/2}},\, f\otimes m \mapsto \sum_i \nb{g\mapsto f([g])\langle m, g.d_i\rangle}\otimes x_i\]
	thus is \(K_0E^{hF_{3/2}}\)-linear.
	The \(E^{hF_{3/2}}\)-linearization of the restriction of \(D'\) to \(MU\langle6\rangle\) will induce the same map in \(K_0\) as \(D'\), so it is still be a \(K(2)\)-local equivalence.
	Thus, we can take \(D\) to be the restriction of \(D'\) to \(MU\langle 6\rangle\).
\end{proof}
\begin{remark}
	The proof of the above also shows that the projections \(D_i=pr_i\circ D\) are given mod \(\mm\) by the unique classes \(d_i\in K^0MU\langle 6\rangle\) such that
	\[\langle m, d_i\rangle = \lambda_i(r_M(m))\]
	for all \(m\in K_0MU\langle 6\rangle\).
\end{remark}

The exact same proof also goes through for \(MString\) in place of \(MU\langle 6\rangle\), so we get:
\begin{theorem}\label{thm:splitMString1}
	There exists a map 
	\[C:MString\to L_{K}\bigoplus_{i\in I} E\]
	such that the free \(E^{hF_{3/2}}\)-linearization of \(C\) induces a \(K\)-equivalence
	\[E^{hF_{3/2}}\otimes MString \xrightarrow{\sim} \bigoplus_{i\in I} E.\]
\end{theorem}

Let us now prove that this splitting is already determined by spin characteristic classes and that it descends to the fixed points under the slightly larger subgroup \(K=\langle F_{3/2},\pi\rangle\).
\begin{theorem}\label{thm:splitMString2}
  There exists a map
  \[Q':BSpin_+\to L_{K(2)}\bigoplus_{j\in J} E\]
  such that the \(E^{hK}\)-linear map
  \[E^{hK}\otimes MString\xrightarrow{\tau_W pr^*Q'}L_{K(2)}\bigoplus_{j\in J} E\]
  is a \(K(2)\)-local equivalence.
\end{theorem}
\begin{proof}
  Let \(M=K_0MString\) and \(N=K_0BSpin\).
  Starting as in the proof of \cref{thm:splitMString1} we now choose a particular retraction:
  Choose a retraction \(r_N\) of \(N_1\subset N\).
  Choose a basis \(\{x_j\}_{j\in J}\) of \(N_1\), let \(\rho_j:N_1\to\ff_4\) be the associated projections, and let \(q_j\in K^0BSpin\) be the unique class such that
  \[\forall x\in K_0BSpin:\rho_j(r_N(x))=\langle x, q_j\rangle.\]
  By \cref{prop:identifyingfixedpoints} we obtain a retraction of the form
  % https://q.uiver.app/#q=WzAsNCxbMCwwLCJNIl0sWzIsMCwiTlxcdGltZXMgTiJdLFsyLDEsIk4gXzFcXHRpbWVzIE5fMSJdLFswLDEsIk1ee0ZfezMvMn19Il0sWzMsMCwiIiwwLHsic3R5bGUiOnsidGFpbCI6eyJuYW1lIjoiaG9vayIsInNpZGUiOiJ0b3AifX19XSxbMCwxLCJldl8xXFx0aW1lcyBldl8xXFxjaXJjXFxwaSJdLFsxLDIsInJfTlxcdGltZXMgcl9OIl0sWzMsMiwiXFxzaW1lcSIsMl1d
  \[\begin{tikzcd}
	  M && {N\oplus N} \\
	  {M^{F_{3/2}}} && {N _1\oplus N_1}
	  \arrow["{ev_1\oplus ev_1\circ\pi}", from=1-1, to=1-3]
	  \arrow["{r_N\oplus r_N}", from=1-3, to=2-3]
	  \arrow[hook, from=2-1, to=1-1]
	  \arrow["\simeq"', from=2-1, to=2-3]
  \end{tikzcd}\]
  The classes in \(K^0MString\) such that pairing \(m\in K_0MString\) with them equals \(\rho_j(r_N(ev_1(m)))\) and \(\rho_j(r_N(ev_1(\pi.m)))\) are given by
  \[\{\tau_Wpr^*q_j,\tau_W\theta^{\pi^{-1}}pr^*\pi^{-1}.q_j\}_{j\in J}\]
  where we identified \(I\) with \(J\coprod J\) in a suitable way.
  Lifting these as above and projecting to the first half of the \(J\)'s we get a map
  \[\tau_W C:MString\to L_{K(2)}\bigoplus_{j\in J}E\]
  such that \(C_j=pr_j\circ C \equiv pr^*q_j\) mod \(\mm\).
  Expanding the \(C_j\) as 
  \[C_j=\sum_{k\geq0}(r-1)^kpr^*C_j^{(k)}\]
  shows that \(C_j^{(0)}\equiv q_j\) mod \(\mm\).

  We now want to show that already the collection \(\tau_Wpr^*C_j^{(0)}\) satisfies the finiteness condition of \cref{lem:manipulatingci}.
  It suffices to show this on the pro-basis \(\alpha_{i,k}\) considered in the proof of \cref{prop:EBstringfunctions}:
  \[\langle \alpha_{i,k},C_j\rangle=\langle y_i,C_j^{(k)}\rangle.\]
  Leaving out the summands \((r-1)^kpr^*C_j^{(k)}\) for \(k>0\) in \(C_j\) will result in more of these pairings being zero, but will not change the values of the non-zero pairings.
  Thus, also the collection \(\tau_W pr^* C_j^{(0)}\) satisfies the finiteness condition of \cref{lem:manipulatingci} and assembles to a map we denote
  \[\tau_Wpr^*C^{(0)}:MString\to L_{K(2)}\bigoplus_{j\in J}E.\]
  
  We claim that the \(E^{hK}\)-linearization of this map induces a \(K(2)\)-equivalence.
  In fact, since the extension \(E^{hK}\to E^{hF_{3/2}}\) is faithful by \cite[Proposition~5.4.9(b)]{Rognes_galois}, we may show that the \(E^{hK}\)-linear map is an equivalence after tensoring up along this extension.
  Consider the \(E^{hF_{3/2}}\)-linear equivalence
  \[E^{hF_{3/2}}\otimes_{E^{hK}}E\to E\oplus E\]
  which is induced by
  \[(\id,\psi^{\pi^{-1}}):E\to E\oplus E.\]
  We obtain an \(E^{hF_{3/2}}\)-linear map
  \[ E^{hF_{3/2}}\otimes MString\to L_{K(2)}\bigoplus_{j\in J}\nb{E\oplus E}\]
  such that the \(j\)th projection, restricted to \(MString\), is given by
  \[\nb{\tau_Wpr^*C_j^{(0)}, \tau_W\theta^{\pi^{-1}}pr^*\pi^{-1}.C_j^{(0)}},\]
  so it is a lift of the map resulting from the Milnor-Moore argument applied to the retraction we chose, and thus an equivalence.
  This shows that we can take \(C^{(0)}\) to be the map \(Q'\).
\end{proof}

\section{Last steps}\label{section:projections}
In this section we will bridge the gap between \cref{thm:splitMString2} and the statements of \cref{maintheorem}.
First we shall take our map
\[Q':BSpin_+\to L_{K(2)}\bigoplus_J E\]
and show that, for the right choice of retraction \(r_N\) in the proof of \cref{thm:splitMString2}, it already factors over a sum of shifts of \(TMF_0(3)\).
The resulting map \(Q\) will then already induce a \(K(2)\)-local equivalence after \(E^{hH}\)-linearization.
We then move on to use our partial calculations of the invariants \(N_1\) and \cref{prop:detect5classes} to prove the 'furthermore' part of \cref{maintheorem}.

\subsection{From \(K\) to \(H\)}
We begin with two lemmas which, essentially, follow from the calculation of \(TMF_0(3)^*BSpin\) in \cite{tmf03} and the determination of the homotopy fixed point spectral sequence for \(TMF_0(3)\) in \cite{MahowaldRezk_level3}.

\begin{lemma}\label{lem:liftToER}
  Let \(ER=E^{h\{\pm 1\}}\).
  Then for any set \(S\) and any integer \(k\) the natural map
  \[\pi_{16k}map\nb{BSpin_+,L_{K(2)}\bigoplus_S ER}\to\pi_{16k}map\nb{BSpin_+,L_{K(2)}\bigoplus_S E} \]
  is an isomorphism.
\end{lemma}

\begin{sseqdata}[ name = HFPSSER, Adams grading, grid = go, xrange={0}{16}, yrange={0}{10}, xscale=0.6, yscale =0.4, no ticks, axes type = frame ]
  \foreach \x in {-16,-12,-8,-4,0,4,8,12,16}{
      \class[rectangle, minimum size = 0.8em](\x,0)
      \foreach \y in {1,2,3,4,5,6,7,8,9,10,11,12,13,14,15,16,17}{
          \class[minimum size = 0.8em](\x+\y,\y)
          \structline
      }
  }
  \foreach \x in {-12,-4,4,12}{
      \d3(\x,0)
      \replacesource[rectangle,"2", minimum size = 0.8em]
      \replacetarget[fill]
      \structline[page = 4](\x-2,2)(\x-1,3)
      \foreach \y in {1,2,3,4,5,6,7,8,9,10,11,12,13,14}{
          \d3(\x+\y,\y)
          \replacetarget[fill]
          \structline[page = 4](\y+\x-2,\y+2)(\y+\x-1,\y+3)
      }
  }
  \foreach \x in {-8,8}{
      \d7(\x,0)
      \replacesource[rectangle, minimum size= 0.8em, "\mm"]
      \d7(\x+1,1)
      \replacesource[ minimum size= 0.8em, "u_1"]
      \d7(\x+2,2)
      \replacesource[ minimum size= 0.8em, "u_1"]
      \foreach \y in {3,4,5,6,7,8,9,10}{
          \d7(\x+\y,\y)
      }
      \structline[page=8](\x,0)(\x+1,1)
      \structline[page=8](\x+1,1)(\x+2,2)
  }
\end{sseqdata}

\begin{proof}
  The proof proceeds by looking at the homotopy fixed point spectral sequence (HFPSS).
  Let us first assume that \(S=\{*\}\).

  Comparing with \cite[§~4]{MahowaldRezk_level3} we can deduce the whole HFPSS for \(ER\):
  The \(E_2\)-page is given by \(E_0[u^{\pm 2},c]/(2c)\)
  where \(E_0\) lies in bidegree \((0,0)\), \(u^2\) lies in \((-4,0)\), and \(c\) lies in \((1,1)\).
  %We also have that \(u_1c\) detects \(\eta\) and \((u_1^3-1)c^3\) detects \(\nu\).
  For degree reasons this is also the \(E_3\)-page, and \(d_3\) is determined by \(d_3(u_1)=0\), \(d_3(c)=0\), and \(d_3(1/u^2)=u_1c^3\).

  Next, the \(E_4\)-page equals the \(E_7\)-page for degree reasons, and there is a further differential \(d_7(1/u^4)=c^7\).
  Finally, the spectral sequence collapses at the \(E_8\)-page with finite vanishing line:
  \begin{center}
    \printpage[ name = HFPSSER, page = 8 ]
  \end{center}
  In this picture a box stands for \(E_0\), a circle for \(E_0/2\), and a dot for \(E_0/\mm\), while a box with a \(2\) inside stands for the ideal \((2)\subset E_0\), and so on.
  We see that the class \(u^8\) lifts (uniquely) to \(ER\), so that \(ER\) is \(16\)-periodic, and that the map \(\pi_{16k}ER\to\pi_{16k} E\) is an isomorphism.

  By \cite[Theorem~1.1]{tmf03} and \cite[Theorem~1.2(ii)]{Laures_TMF13BSpin} there are classes
  \[\pi_i\in TMF_0(3)^{-32i}BSpin\]
  such that \(E^*BSpin = E^*[\![\pi_1,\pi_2,\ldots]\!]\).
  Let \(\tilde{\pi}_i=u^{16i}\pi_i\in E^0BSpin\).
  By the above discussion we see that \(\tilde{\pi}_i\) lifts to \(ER^0BSpin\) and thus is a permanent cycle in the HFPSS for \(ER^*BSpin\), as are all monomials in the \(\tilde{\pi}_i\).
  The HFPSS for \(ER^*BSpin\) thus looks like a product of copies of the one for \(ER\) and we find
  \[ER^*BSpin=ER^*[\![\tilde{\pi}_1,\tilde{\pi}_2,\ldots]\!].\]
  This finishes the case \(S=\{*\}\).

  For general \(S\) first consider the case where we take the product over \(S\) instead of the sum.
  Here the HFPSS is then also just a product over the above HFPSS for \(S=\{*\}\).
  On coefficients the map
  \[L_{K(2)}\bigoplus_SE\to \prod_S E\]
  induces a split injection of \(E_*\)-modules by \cite[Proposition~A.13]{666}.
  Since the differentials in the HFPSS for \(ER\) are all given by multiplication with elements in \(E_0\), the comparison map between the spectral sequences for the sum and product is a split injection on every page and bidegree as well.
  This allows us to transport the differentials to the case at hand, proving the lemma.
\end{proof}
\begin{lemma}\label{lem:manipulatingQj}
  For any set \(S\) and any \(n_s\in\{0,16,32\}\) the natural map
  \[\pi_0Map\nb{BSpin_+,L_{K(2)}\bigoplus_{s\in S}\Sigma^{n_s} TMF_0(3)}\to\prod_{s\in S}TMF_0(3)^{n_s}BSpin\]
  is injective.
\end{lemma}
\begin{proof}
  Combining \cref{lem:manipulatingci} and \cref{lem:liftToER} shows that the statement is true for \(TMF_0(3)\) replaced with \(ER\).
  The homotopy fixed point spectral sequence for the action of \(\omega\) collapses on the \(E_2\)-page and is concentrated on the \(0\)-line, thus the statement is also true for \(E^{h\langle -1,\omega\rangle}\).
  Finally, we have an equivalence
  \[TMF_0(3)\oplus TMF_0(3)\xrightarrow{(1,\zeta)}E^{h\langle-1,\omega\rangle}\]
  so the statement follows.
\end{proof}

Having this in our hands we can now descend along the ladder of extensions
\[K\subset \langle K,-1\rangle\subset\langle K,-1,\omega\rangle\subset\langle K,-1,\omega,\sigma\rangle=H.\]
\begin{theorem}\label{thm:splitMString5}
  There exists a map
  \[Q:BSpin_+\to L_{K(2)}\bigoplus_{j\in J}\Sigma^{n_j}TMF_0(3),\]
  with \(n_j\in\{0,16,32\}\), such that the \(E^{hH}\)-linear map
  \[E^{hH}\otimes MString\xrightarrow{\tau_W pr^*Q}\bigoplus_{j\in J}\Sigma^{n_j}TMF_0(3)\]
  is a \(K(2)\)-local equivalence.
\end{theorem}
\begin{proof}
  The proof works by observing that, for suitable choices of retraction and basis, the map \(Q'\) constructed in \cref{thm:splitMString2} can be made to factor as claimed.

  Let us begin by choosing a particular retraction \(r_N:N\to N_1\) and basis of \(N_1\).
  Since \(\ff_4\) contains a third root of unity \(\zeta\) we can decompose \(N_1\) as 
  \[N_1=L_0\oplus L_8\oplus L_{16},\,L_i=\{x\in N_1\mid \omega.x=\zeta^i x\}\]
  and similarly for \(N=G_0\oplus G_8\oplus G_{16}\); the choice of \(\{0,8,16\}\) instead of \(\{0,1,2\}\) will allow us to use \cref{lem:liftToER} below.
  Since the action of \(\sigma\) is \(\ff_4\)-antilinear and we have that \(\sigma\omega\sigma=\omega^2\), \(L_i\) and \(G_i\) are \(\sigma\)-invariant subspaces.
  Thus there are canonical equivalences
  \[L_i=\ff_4\otimes_{\ff_2}L_i^\sigma\text{ and }G_i=\ff_4\otimes_{\ff_2}G_i^\sigma.\]
  Now choose retractions \(r_i:G_i^\sigma \to L_i^\sigma\) and let \(r_N=\id_{\ff_4}\otimes_{\ff_2}(r_0\oplus r_8\oplus r_{16})\), so that \(r_N\) intertwines the action of \(\omega\) and \(\sigma\).
  Also choose a basis \(\{x_j\}_{j\in J}\) of \(N_1\) such that each \(x_j\) lies in one of the \(L_i^\sigma\), say with \(i=m_j\in\{0,8,16\}\).

  Let \(\rho_j:N_1\to\ff_4\) be the projections such that
  \[\forall x\in N_1:x=\sum_{j\in J}\rho_j(x)x_j\]
  and \(q_j\in K^0BSpin\) the unique class such that
  \[\forall x\in N: \rho_j(r_N(x))=\langle x, q_j\rangle.\]
  By construction, we have that \(\omega.q_j=\zeta^{2m_j}q_j\) and \(\sigma.q_j=q_j\), so that the classes \(u^{m_j}q_j\in K^{2m_j}BSpin\) are invariant under the action of \(\omega\) and \(\sigma\).

  In \cref{thm:splitMString2} we constructed a map
  \[Q':BSpin_+\to L_{K(2)}\bigoplus_{j\in J}E\]
  such that \(pr_j\circ Q'\equiv q_j\) mod \(\mm\).
  Composing with the direct sum of maps \(u^{m_j}:E\to \Sigma^{2m_j}E\) we obtain a map
  \[Q'':BSpin_+\to L_{K(2)}\bigoplus_{j\in J}\Sigma^{2m_j}E\]
  such that \(pr_j\circ Q''=u^{m_j}q_j\).
  By \cref{lem:liftToER} there now exists a unique map
  \[Q''':BSpin_+\to L_{K(2)}\bigoplus_{j\in J}\Sigma^{2m_j}ER\]
  lifting \(Q''\).
  Consider the map
  \[\frac{Q'''+\omega.Q'''+\omega^2.Q'''}{3}:BSpin_+\to L_{K(2)}\bigoplus_{j\in J}\Sigma^{2m_j}ER.\]
  Since \(3\) is invertible the homotopy fixed point spectral sequence for the action of \(\omega\) collapses at the \(E_2\)-page and is concentrated on the zeroth line.
  Thus, the above map, which is \(\omega\)-invariant by construction, admits a unique lift to
  \[Q'''':BSpin_+\to L_{K(2)}\bigoplus_{j\in J}\Sigma^{2m_j}E^{h\langle -1,\omega\rangle}.\]
  
  Finally, we have a commutative square
  % https://q.uiver.app/#q=WzAsNCxbMCwwLCJUTUZfMCgzKVxcb3BsdXMgVE1GXzAoMykiXSxbMSwwLCJFXntoXFxsYW5nbGUtMSxcXG9tZWdhXFxyYW5nbGV9Il0sWzAsMSwiRV57aFxcbGFuZ2xlXFxzaWdtYVxccmFuZ2xlfVxcb3BsdXMgRV57aFxcbGFuZ2xlXFxzaWdtYVxccmFuZ2xlfSJdLFsxLDEsIkUiXSxbMCwxLCIoMSxcXHpldGEpIl0sWzAsMl0sWzIsM10sWzEsM11d
  \[\begin{tikzcd}
	  {TMF_0(3)\oplus TMF_0(3)} & {E^{h\langle-1,\omega\rangle}} \\
  	{E^{h\langle\sigma\rangle}\oplus E^{h\langle\sigma\rangle}} & E
  	\arrow["{(1,\zeta)}", from=1-1, to=1-2]
  	\arrow[from=1-1, to=2-1]
  	\arrow[from=1-2, to=2-2]
	  \arrow["{(1,\zeta)}", from=2-1, to=2-2]
  \end{tikzcd}\]
  where the horizontal maps are equivalences and the vertical maps inclusions of fixed points.
  Denote by \(\alpha: E\to E^{h\langle\sigma\rangle}\) and \(\beta:E^{h\langle -1,\omega\rangle}\to TMF_0(3)\) the projections to the first summands in the above diagram, which are retractions of the inclusions \(\iota_\alpha:E^{h\langle\sigma\rangle}\to E\) and \(\iota_\beta:TMF_0(3)\to E^{h\langle -1,\omega\rangle}\).

  Let \(Q\) be the composite of \(Q''''\) with a direct sum of shifts of \(\beta\).
  We claim that the \(E^{hH}\)-linear map
  \[E^{hH}\otimes MString\xrightarrow{\tau_Wpr^*Q}L_{K(2)}\bigoplus_{j\in J}\Sigma^{2m_j}TMF_0(3)\]
  is a \(K(2)\)-local equivalence.
  It is enough to show that it is an equivalence after tensoring up along the extension \(E^{hH}\to E^{hK}\), which is faithful by \cite[Proposition~5.4.9(b)]{Rognes_galois}.
  Note that we have an \(E^{hK}\)-linear \(K(2)\)-local equivalence
  \[E^{hK}\otimes_{E^{hH}}TMF_0(3)\xrightarrow{\sim}E\]
  which is induced by the inclusion of fixed points \(\iota:TMF_0(3)\to E\).
  So after tensoring up the above map will be induced by the composite
  \[BSpin_+\xrightarrow{Q}L_{K(2)}\bigoplus_{j\in J}\Sigma^{2m_j}TMF_0(3)\xrightarrow{\iota}L_{K(2)}\bigoplus_{j\in J}\Sigma^{2m_j}E.\]
  Tracing through the construction we obtain
  \begin{align*}
    pr_j\circ\iota\circ Q&=\iota_\alpha\circ \alpha\circ\frac{\id +\psi^{\omega}+\psi^{\omega^2}}{3}\circ u^{m_j}\circ pr_j\circ Q'\\
    &\equiv\iota_\alpha\circ \alpha\circ\frac{\id +\psi^{\omega}+\psi^{\omega^2}}{3}\circ u^{m_j}q_j+\mm\\
    &\equiv \iota_\alpha\circ \alpha\circ u^{m_j}q_j+\mm\\
    &\equiv u^{m_j}q_j+\mm
  \end{align*}
  where we used in the third step that \(u^{m_j}q_j\) is \(\omega\)-invariant and in the fourth step that it is \(\sigma\)-invariant so that the idempotent \(\iota_\alpha\circ\alpha\) acts trivially on it.
  Thus, after tensoring up and shifting the summands, \(Q\) induces the same map as \(Q'\) in \(K_0\), thus it induces a \(K(2)\)-local equivalence.
\end{proof}

\subsection{Determining some \(Q_j\)}
By \cref{lem:manipulatingci} and \cref{lem:manipulatingQj} it is relatively easy to determine finitely many of the \(Q_j\).
This we will do with the following strategy which we have already seen in the proof of \cref{thm:splitMU6}:
\begin{construction}\label{cons:finitelyManyQj}
  Let \(V\subset N_1\) be a \(k\)-dimensional subspace, and let \(c_i\) be classes in  \(K^0BSpin\) such that
  \[\epsilon:N\xrightarrow{x\mapsto\langle x,c_i\rangle}\bigoplus_{i=1}^k\ff_4\]
  is an isomorphism when restricted to \(V\).
  Then we have natural isomorphisms
  \[N=V\oplus \ker\epsilon,\, N_1=V\oplus \ker \epsilon\vert_{N_1}\]
  and we may choose a retraction of the form \(r_N=\id_V\oplus r_{\ker}\).
  It follows that, if we choose a basis of \(N_1\) as a basis of \(V\) dual to the \(c_i\) plus a basis of \(\ker\epsilon\vert_{N_1}\), the \(Q_j\) constructed in \cref{thm:splitMString5} associated to the basis elements of \(V\) are lifts of the \(c_i\), and since they are finitely many\footnote{For infinitely many we would need a version of \cref{lem:manipulatingci} for \(TMF_0(3)\) and be able to verify the corresponding finiteness conditions.} we may replace them by any other lifts of the \(c_i\).
  Also note that the other \(Q_j\) must pair to \(0\) mod \(\mm\) with all elements in \(V\).
  Of course, in all of this we must make sure to satisfy the invariance conditions on \(r_N\) and the basis of \(N_1\) we used in the proof of \cref{thm:splitMString5}.
  This will be the case whenever the classes \(c_i\) are \(\sigma\)-invariant and satisfy \(\omega.c_i=\zeta^{2m_i}c_i\).
\end{construction}

This turns determining (finitely many of) the \(Q_j\) into a game of listing elements in \(N_1\) and detecting them with suitable classes in \(TMF_0(3)^*BSpin\).

\begin{proposition}\label{prop:makeunital}
  In the first part of \cref{maintheorem} we may choose one of the \(Q_j\) to be \(1\in TMF_0(3)^0BSpin\), and the other \(Q_j\)s to be reduced classes.
  It follows that the Ando-Hopkins-Rezk orientation \(MString\to TMF\) admits, after tensoring with \(E^{hH}\), an \(E^{hH}\)-linear and unital section.
\end{proposition}
\begin{proof}
  Since the Morava stabilizer group acts on \(E\) by ring maps it fixes the unit \(1\in M^{F_{3/2}}\subset K_0MString\).
  In particular, as \(1=pr_*1\), we also have that \(1\in N_1\subset K_0BSpin\), which we may detect with \(1\in TMF_0(3)^0BSpin\).
  By the discussion in \cref{cons:finitelyManyQj} we get from \cref{thm:splitMString5} an \(E^{hH}\) linear equivalence
  \[E^{hH}\otimes MString\to\bigoplus_{j\in J}\Sigma^{n_j} TMF_0(3)\]
  such that, projecting onto the summand corresponding to the unit, we are looking at the map
  \[E^{hH}\otimes MString\xrightarrow{\id\otimes\tau_W}E^{hH}\otimes TMF=TMF_0(3)\]
  which we identify with the Ando-Hopkins-Rezk orientation \(MString\to TMF\) tensored with the identity on \(E^{hH}\).
  This already shows that, after tensoring with \(E^{hH}\), the Ando-Hopkins-Rezk orientation admits an \(E^{hH}\)-linear section, by including the summand back into the sum and applying the inverse equivalence.

  By \cref{lem:manipulatingQj} the section is then  unital if and only if all the other \(Q_j\) are reduced classes.
  Looking at the process we went through to extract \(Q\) out of \(Q'\) we see that it is enough to make sure that the other \(Q'_j\) are reduced classes.
  This can be arranged: give the 'first' summand the index \(0\in J\).
  Then the collection \(\{Q'_j\vert_{pt}\in E^0\}_{j\neq 0}\) also satisfies the finiteness condition of \cref{lem:manipulatingci}, and so does the collection \(\{Q'_j-Q'_j\vert_{pt}\}_{j\neq 0}\).
  So we may replace all other \(Q'_j\) by their reduced counterparts.
  Since we have chosen \(q_0=1\in K^0BSpin\) the construction of the discussion above tells us that the other \(q_j\) are reduced, so the reductions of the other \(Q'_j\) are still lifts of the \(q_j\).
\end{proof}

We now want to discuss what summands of \(\Sigma^{2m} TMF_0(3)\) in \(MString\) would look like in this splitting.
The inclusion map \(\Sigma^{2m} TMF_0(3)\to MString\) would map the unit of \(TMF_0(3)\) to a class \(a\in N_1\) such that \(\omega.a=\zeta^ma\), \(\sigma.a=a\), and such that \(\{a,i.a,j.a,k.a\}\) are linearly independent.
Changing the basis a bit and denoting
\[x=i+j+k,y=i+\zeta j+\zeta^2k,z=i+\zeta^2j+\zeta k\in\ff_4[Q_8/\{\pm1\}]\]
this is equivalent to asking for \(\{a,x.a,y.a,z.a\}\) to be linearly independent.
This choice is convenient since, in the terminology of the proof of \cref{thm:splitMString5}, we then have that
\[a\in L_m^\sigma,x.a\in L_m^\sigma, y.a\in L_{m+2}^\sigma,\text{ and }z.a\in L_{m+1}^\sigma.\]

On the other hand, after tensoring with \(E^{hH}\), \(TMF_0(3)\) splits as four copies of \(TMF_0(3)\):
\[E^{hH}\otimes TMF_0(3)\xrightarrow{\sim}TMF_0(3)\oplus TMF_0(3)\oplus \Sigma^{32}TMF_0(3)\oplus \Sigma^{16}TMF_0(3)\]
where the equivalence is \(E^{hH}\)-linear and the projection maps are, after restricting to \(TMF_0(3)\) in the source and including as fixed points into \(E\) in the target, given by
\[\iota,(\psi^i+\psi^j+\psi^k)\circ\iota,u^{16}(\psi^i+\zeta \psi^j+\zeta^2\psi^k)\circ\iota,\text{ and }u^{8}(\psi^i+\zeta^2 \psi^j+\zeta \psi^k)\circ\iota.\]
Looking for such summands corresponds to looking for classes \(a\in L_m^\sigma\) and \(p\in TMF_0(3)^{2m}BSpin\) such that pairing the classes \(\{a,x.a,y.a,z.a\}\) with the classes \(\{p,x.p,y.p,z.p\}\) gives an invertible matrix (where we abused notation to identify \(p\) with its image in \(K^0BSpin\)).
A little computation shows that the determinant of the matrix is given by
\[\det\nb{\langle\{a,x.a,y.a,z.a\}, \{p,x.p,y.p,z.p\}\rangle}=\langle a,p\rangle+\langle x.a,p\rangle.\]
\begin{proposition}\label{prop:finalSummand}
  In the setting of \cref{thm:splitMString5} one can choose \(Q\) such that the map 
  \[MString\xrightarrow{\tau_W\oplus\tau_W pr^*p}TMF\oplus\Sigma^{16}TMF_0(3)\]
  tensored with the identity on \(E^{hH}\) appears as a direct summand, where
  \[p=\pi_4/\Delta^6\in TMF_0(3)^{16}BSpin.\]
  It thus admits an \(E^{hH}\)-linear section, and, restricted to \(E^{hH}\otimes TMF\), this section can be made unital.
\end{proposition}
\begin{proof}
  With the strategy described in \cref{cons:finitelyManyQj}, consider the subspace
  \[V=\mathrm{span}\{1,cd^3,i.(cd^3),j.(cd^3),k.(cd^3)\}\subset N_1\]
  and the detecting classes \(1,p,x.p,y.p,z.p\).
  The resulting map
  \[Q:BSpin_+\to L_{K(2)}\bigoplus_{j\in J}\Sigma^{n_j}TMF_0(3)\]
  will then have \(Q_0,\ldots,Q_4\) as lifts of \(1,p,x.p,y.p,z.p\), and by the discussion above we may replace them by any other lifts.
  By the discussion about the splitting of \(E^{hH}\otimes TMF_0(3)\) and \cref{prop:makeunital} we thus see that we may choose these as the composite
  \begin{align*}
    &BSpin_+\xrightarrow{1\oplus\pi_4/\Delta^6}TMF\oplus\Sigma^{16}TMF_0(3)\to E^{hH}\otimes\nb{TMF\oplus \Sigma^{16}TMF_0(3)}\\
    &\xrightarrow{\sim}TMF_0(3)\oplus\Sigma^{16}TMF_0(3)\oplus\Sigma^{16}TMF_0(3)\oplus\Sigma^{48}TMF_0(3)\oplus\Sigma^{32}TMF_0(3).
  \end{align*}
  We then, by the same argument as in the proof of \cref{prop:makeunital}, make sure that all classes but \(Q_0\) are reduced, which does not affect the classes coming from \(\pi_4/\Delta^6\) since they are already reduced.
\end{proof}

\appendix
\section{The Milnor-Moore argument}\label{section:minlorMoore}
In this Section we recall some basic definitions around coalgebras and their comodules, and prove a slight strengthening of the Minor-Moore type argument introduced in \cite{mm}.

Let \(\ff\) be a field, \(C\) be a coalgebra over \(\ff\), and \(M\) a (left) \(C\)-comodule.
\begin{definition}
	The {\it coradical} \(F_0(C)\) of \(C\) is defined to be the sum of all simple subcoalgebras of \(C\).
	We say that \(C\) is {\it pointed} if each simple subcoalgebra is one dimensional over \(\ff\).

	The {\it coradical filtration} \(F_n(C)\) of \(C\) is inductively defined by
	\[F_{n+1}(C):=\Delta^{-1}\nb{F_n(C)\otimes C + C\otimes F_0(C)}.\]
	This also induces a natural filtration of the comodule \(M\):
	\[F_n(M):=\psi^{-1}\nb{F_n(C)\otimes M}.\]
\end{definition}
\begin{proposition}[{\cite[cor.~9.0.4, cor.~9.1.7]{Sweedler_HopfAlgebras},}]
	The coradical filtration is exhaustive, that is
	\[C=\bigcup_{n\geq 0}F_n(C).\]
	Moreover, this makes \(C\) into a filtered coalgebra, that is
	\[\Delta(F_n(C))\subset \sum_{k=0}^nF_{n-k}(C)\otimes F_k(C).\]
\end{proposition}
\begin{remark}
	This also shows that the coradical filtration of \(C\) coincides with the natural filtration as a comodule over itself.
\end{remark}

Before we proceed, we will need to recall some constructions around the {\it dual algebra} of \(C\).
\begin{construction}\label{cons:dualAlgebra}
	Let \(C^\vee=\hm_\ff(C,\ff)\) and equip it with the \(\ff\)-algebra structure
	\[\phi\star\psi(c):=(\phi\otimes\psi)(\Delta(c)),\, 1:=\epsilon.\]
	We also equip it with the topology described in \cite[§~I.2]{Dieudonne_IntroFG}.
	For linear subspaces \(L\subset C\) and \(\Phi\subset C^\vee\), we call
	\[L^\perp=\{\phi\in C^\vee\mid \phi\vert_L=0\}\text{ and }\Phi^\perp=\{c\in C\mid ev_c\vert_\Phi=0\}\]
	the {\it annihilator spaces} of \(L\) and \(\Phi\).
	Then \(L^\perp\) is closed, \(L=(L^\perp)^\perp\), and \((\Phi^\perp)^\perp\) is the closure of \(\Phi\) in \(C^\vee\).
	This gives an inclusion reversing bijection between subspaces of \(C\) and closed subspaces of \(C^\vee\), and restricts to an inclusion reversing bijection between subcoalgebras and ideals.

	Note that a maximal closed ideal is a maximal ideal in \(C^\vee\): if \(I\) contains a maximal closed ideal \(\mm\), then \(C^\vee/I\) is a quotient of \(C^\vee/\mm\).
	Per definition the projection \(C^\vee\to C^\vee/\mm\) is continuous for the discrete topology on the target.
	Thus, also the projection \(C^\vee\to C^\vee/I\) is continuous, showing that \(I\) is closed.

	Let \(J\subset C^\vee\) be the Jacobson radical, which is the intersection of all maximal closed ideals by the previous paragraph, so itself and all its powers are closed.
	We also have that \(F_n(C)=(J^{n+1})^\perp\).

	Consider the right action of \(C^\vee\) on \(M\) given by
	\[m.\phi:=(\phi\otimes \id)(\psi(m)).\]
	We can then identify \(F_n(M)\) with the \(J^{n+1}\)-torsion submodule \(M[J^{n+1}]\).
\end{construction}

\begin{lemma}
	Let \(M\) be a \(C\)-comodule and \(m\in F_n(M)\).
	Then
	\[\psi(m)\in \sum_{k=0}^nF_{n-k}(C)\otimes F_k(M).\]
\end{lemma}
\begin{proof}
	Let \(\{c_i^n\in F_n(C)\}_{i\in I_n,n\geq0}\) be such that \(\{c_i^k\mid k\leq n\}\) is a basis of \(F_n(C)\), and similarly \(\{m_j^n \in F_n(M)\}_{j\in J_n,n\geq 0}\).
	Let also \(\phi^n_i\) be the linear forms dual to the \(c^n_i\).
	In particular, since \(\phi^n_i\) vanishes on \(F_{n-1}(C)\), it lies in \(J^n\subset C^\vee\).
	Write \(\psi(m)\) as
	\[\psi(m)=\sum_{a,i,b,j}D(m,a,i,b,j)c^a_i\otimes m^b_j.\]
	This shows that
	\[m.\phi^a_i=\sum_{b,j}D(m,a,i,b,j)m^b_j.\]
	As \(m\) was \(J^{n+1}\)-torsion by assumption, we must have that \(m.\phi^a_i\) is \(J^{n+1-a}\)-torsion, so the coefficients \(D(m,a,i,b,j)\) with \(a+b>n\) must vanish, and we are done.
\end{proof}

From now on, let \(C\) be a Hopf algebra over \(\ff\), which, as a coalgebra, is pointed.
Let \(G\subset C\) be the subset of group-like elements of \(C\).
Then the multiplication of \(C\) equips this with the structure of a group, and pointedness implies that \(F_0(C)\simeq \ff[G]\) as Hopf algebras.
\begin{definition}
	The primitive elements of \(M\) are defined as
	\[P_1(M):=\{m\in M\mid \psi(m)=1\otimes M\}.\]
\end{definition}

\begin{lemma}\label{lem:F0=FGP1}
	Let \(M\) be an \(F_0(C)=\ff[G]\)-module in \(C\)-comodules.
	Then we have that the map
	\[F_0(C)\otimes P_1(M)\to F_0(M),\,g\otimes m\mapsto g.m\]
	is an equivalence.
\end{lemma}
\begin{proof}
	We explicitly construct an inverse:
	for \(m\in F_0(M)\) write \(\psi(m)=\sum_g g\otimes m_g\).
	By coassociativity we have that \(\psi(m_g)=g\otimes m_g\), so that \(g^{-1}.m_g\in P_1(M)\).
	The inverse then maps
	\[m\mapsto\sum_gg\otimes g^{-1}.m_g.\]
\end{proof}
\begin{definition}\label{def:splittable}
	Let \(M\) be an \(\ff[G]\)-module in \(C\)-comodules.
	Consider the natural maps
	\[q_n:F_n(M)/F_{n-1}(M)\to F_n(C)/F_{n-1}(C)\otimes P_1(M),\,[m]\mapsto \sum_{i,g}[c_i]\otimes g^{-1}.m_{i,g}\]
	for \(\psi(m)\equiv \sum_ic_i\otimes m_i\) mod \(F_{n-1}(C)\otimes F_n(M)\) with \(c_i\in F_n(C)\) and \(m_i\in F_0(M)\).
	We call \(M\) {\it splittable} if these are surjective for all \(n\geq 0\).
\end{definition}
\begin{lemma}
	Let \(M\) be an \(\ff[G]\)-module in \(C\)-comodules.
	Then the natural maps
	\[q_n:F_n(M)/F_{n-1}(M)\to F_n(C)/F_{n-1}(C)\otimes P_1(M)\]
	are injective for all \(n\geq 0\).
\end{lemma}
\begin{proof}
	The case \(n=0\) follows from Lemma A.12.
	For \(n\geq 1\), consider \(m\in F_n(M)\) and write \(\psi(m)\equiv\sum_{i,g}c_i\otimes m_{i,g}\) mod \(F_{n-1}(C)\otimes F_n(M)\) and \(\Delta(c_i)\equiv\sum_{g}c_{i,g}\otimes g\) mod \(F_{n-1}(C)\otimes F_n(C)\).
	By coassociativity the elements \(c_{i,g}\) fulfill
	\[\Delta(c_{i,g})\equiv c_{i,g}\otimes g\text{ mod }F_{n-1}(C)\otimes F_n(C).\]
	Again by coassociativity we find that
	\[\sum_{i,g}c_i\otimes g\otimes m_{i,g}\equiv \sum_{i,g,h}c_{i,h}\otimes h\otimes m_{i,g}\text{ mod }F_{n-1}(C)\otimes F_0(C)\otimes F_0(M).\]
	Projecting to the terms of the form \(A\otimes h\otimes B\) for a fixed \(h\in G\) we get
	\[\forall{h\in G}:\sum_ic_i\otimes m_{i,h}\equiv \sum_{i,g}c_{i,h}\otimes m_{i,g}\text{ mod }F_{n-1}(C)\otimes F_0(M).\]
	By the equivalence of Lemma A.12 this implies
	\[\circledast\quad \forall h\in G:\sum_ic_i\otimes m_{i,h}\equiv\sum_ic_{i,h}\otimes m_{i,h}\text{ mod }F_{n-1}(C)\otimes F_0(M).\]
	Now assume that
	\[q_n(m)=\sum_{i,g}[c_{i,g}]\otimes g^{-1}.m_{i,g}=0\iff \sum_{i,g}c_{i,g}\otimes g^{-1}.m_{i,g}\in F_{n-1}(C)\otimes P_1(M).\]
	Applying \(\Delta\otimes\id\) this gives
	\[\sum_{i,g}c_{i,g}\otimes g\otimes g^{-1}.m_{i,g}\in F_{n-1}(C)\otimes F_0(C)\otimes P_1(M).\]
	Now Lemma A.12 implies that
	\[\sum_{i,g}c_{i,g}\otimes m_{i,g}\in F_{n-1}(C)\otimes F_0(M).\]
	Using the congruence \(\circledast\) this shows
	\[\psi(m)\equiv\sum_i c_i\otimes m_i\equiv \sum_{i,g}c_{i,g}\otimes m_{i,g}\equiv 0\text{ mod }F_{n-1}(C)\otimes F_n(M).\]
	Thus \(m\) is already in \(F_{n-1}(M)\), and we are done.
\end{proof}

We now prove a slight strengthening of the Milnor-Moore type argument introduced in \cite{mm}.
\begin{theorem}\label{thm:milnorMoore}
	Let \(C\) be a Hopf algebra over \(\ff\) which, as a coalgebra, is pointed, and let \(G\subset C\) be the set of group-like elements.
	Let \(M\) be an \(\ff[G]\)-module in \(C\)-comodules which is splittable.
	Then, for any \(\ff\)-linear retraction \(r:M\to \ff[G]\otimes P_1(M)\) of the natural map \(\ff[G]\otimes P_1(M)\simeq F_0(M)\subset M\), the map
	\[h_r:M\xrightarrow{\psi}C\otimes M\xrightarrow{\id\otimes r}C\otimes\ff[G]\otimes P_1(M)\xrightarrow{\id\otimes\epsilon\otimes\id}C\otimes P_1(M)\]
	is an isomorphism of \(C\)-comodules and exhibits \(M\) as a (co)free \(\ff[G]\)-module in \(C\)-comodules.
\end{theorem}
\begin{proof}
	First of all, \(h_r\) is a map of comodules since already \(\psi:M\to C\otimes M\) is a map of comodules when we give the target the comodule structure \(\Delta\otimes\id\).

	Since the coradical filtration of \(C\), and thus also the natural comodule filtration induced by it, is exhaustive, it suffices to show that \(F_n(h_r)\) is an isomorphism for all \(n\geq0\).
	We will do so by induction in \(n\).

	For \(n=0\) the map \(F_0(h_r)\) equals the map considered in Lemma A.12 since \(r\) is a retraction, so it is an isomorphism.

	Now assume that we have already established that \(F_n(h_r)\) is an isomorphism.
	Consider the diagram with exact rows
	% https://q.uiver.app/#q=WzAsMTAsWzAsMCwiMCJdLFsxLDAsIkZfbihNKSJdLFsyLDAsIkZfe24rMX0oTSkiXSxbMywwLCJGX3tuKzF9KE0pL0ZfbihNKSJdLFs0LDAsIjAiXSxbMSwxLCJGX24oQylcXG90aW1lcyBQXzEoTSkiXSxbMiwxLCJGX3tuKzF9KEMpXFxvdGltZXMgUF8xKE0pIl0sWzMsMSwiRl97bisxfShDKS9GX24oQylcXG90aW1lcyBQXzEoTSkiXSxbMCwxLCIwIl0sWzQsMSwiMCJdLFswLDFdLFsxLDJdLFsyLDNdLFszLDRdLFs4LDVdLFs1LDZdLFs2LDddLFs3LDldLFsxLDUsIkZfbihoX3IpIiwyXSxbMiw2LCJGX3tuKzF9KGhfcikiLDJdLFszLDddXQ==
\[\begin{tikzcd}
	0 & {F_n(M)} & {F_{n+1}(M)} & {F_{n+1}(M)/F_n(M)} & 0 \\
	0 & {F_n(C)\otimes P_1(M)} & {F_{n+1}(C)\otimes P_1(M)} & {F_{n+1}(C)/F_n(C)\otimes P_1(M)} & 0
	\arrow[from=1-1, to=1-2]
	\arrow[from=1-2, to=1-3]
	\arrow["{F_n(h_r)}"', from=1-2, to=2-2]
	\arrow[from=1-3, to=1-4]
	\arrow["{F_{n+1}(h_r)}"', from=1-3, to=2-3]
	\arrow[from=1-4, to=1-5]
	\arrow[from=1-4, to=2-4]
	\arrow[from=2-1, to=2-2]
	\arrow[from=2-2, to=2-3]
	\arrow[from=2-3, to=2-4]
	\arrow[from=2-4, to=2-5]
\end{tikzcd}\]
	We claim that the right vertical map is exactly the natural map \(q_{n+1}\).
	Indeed, consider an element \(m\in F_{n+1}(m)\) and write \(\psi(m)\equiv \sum c_i\otimes m_i\) mod \(F_n(C)\otimes F_{n+1}(M)\).
	Then \(h_r(m)\) is given by
	\[h_r(m)\equiv \sum_{i,g}c_i\otimes g^{-1}.m_{i,g}\text{ mod }F_n(C)\otimes P_1(M)\]since \(r\) is a retraction.
	So the image of \(m\) in the bottom right corner is
	\[\sum_{i,g}[c_i]\otimes g^{-1}.m_{i,g}=q_{n+1}([m]).\]
	Now by Lemma A.14 and the assumption that \(M\) is splittable \(q_{n+1}\) is an isomorphism, so that also \(F_{n+1}(h_r)\) is an isomorphism.
\end{proof}

Let us now connect this with the version stated in \cite{mm}.
\begin{definition}
	Let \(M\) be an \(\ff[G]\)-algbra in \(C\)-comodules.
	We define, for \(n\geq0\), the {\it \(n\)-th graded right primitives} of \(M\) as
	\[\overline{P}_1Gr_n(M):=\{m\in F_n(M)\mid \psi(m)\equiv c\otimes 1\text{ mod } F_{n-1}(M)\otimes F_n(M)\}.\]
	We call a map of \(\ff[G]\)-algebras in \(C\)-comodules \(A:M\to N\) {\it \(\star\)-surjective} if it induces surjections on all graded right primitives.
\end{definition}

\begin{proposition}
	Let \(M\) be an \(\ff[G]\)-algebra in \(C\)-comodules which admits a \(\star\)-surjective map \(A:M\to C\).
	Then \(M\) is splittable.
\end{proposition}
\begin{proof}
	We need to prove that the natural maps
	\[q_n:F_n(M)/F_{n-1}(M)\to F_n(C)/F_{n-1}(C)\otimes P_1(M)\]
	are surjective.
	Since every element \([c]\in F_{n}(C)/F_{n-1}(C)\) can be written as a sum \([c]=\sum_g[c_g]\) with \(\Delta(c_g)\equiv c_g\otimes g\) mod \(F_{n-1}(C)\otimes F_n(C)\) we only need to show that elements of the form \([c_g]\otimes m\) are in the image.
	By assumption, \(c_gg^{-1}\) is an \(n\)-th graded right primitive of \(C\).
	By \(\star\)-surjectivity there exists an \(n\)-th graded right primitive \(a\) of \(M\) such that \(A(a)=c_gg^{-1}\).
	Since \(A\) is a map of comodules, we find that \(\psi(a)\equiv c_gg^{-1}\otimes1\) mod \(F_{n-1}(C)\otimes F_n(M)\).
	Then, using the algebra structure of \(M\), we calculate \(\psi(agm)\equiv c_g\otimes gm\) mod \(F_{n-1}(C)\otimes F_n(M)\).  
	This shows that \(q_n(agm)=[c_g]\otimes m\) and we are done.
\end{proof}

Finally, let us state another criterion for \(\ff[G]\)-algebras in \(C\)-comodules to be splittable
\begin{proposition}\label{prop:splittableCriterion}
	Let \(M\) be an \(\ff[G]\)-algebra in \(C\)-comodules, and assume there exists a map \(s:C\to M\) of comodules with \(s(1)=1\).
	Then M is splittable.
\end{proposition}
\begin{proof}
	We need to prove that the natural maps
	\[q_n:F_n(M)/F_{n-1}(M)\to F_n(C)/F_{n-1}(C)\otimes P_1(M)\]
	are surjective.
	Since every element \([c]\in F_{n}(C)/F_{n-1}(C)\) can be written as a sum \([c]=\sum_g[c_g]\) with \(\Delta(c_g)\equiv c_g\otimes g\) mod \(F_{n-1}(C)\otimes F_n(C)\), we only need to show that elements of the form \([c_g]\otimes m\) are in the image.
	Since \(s\) is a map of comodules with \(s(1)=1\) we have
	\[\psi(s(c_gg^{-1}))=(\id\otimes s)(\Delta(c_gg^{-1}))\equiv(\id\otimes s)(c_gg^{-1}\otimes 1)\equiv s(c_gg^{-1})\otimes 1 \text{ mod }F_{n-1}(C)\otimes F_n(M).\]
	Using the algebras structure of \(M\) we calculate
	\[\psi(s(c_gg^{-1})gm)\equiv c_g\otimes gm\text{ mod }F_{n-1}(C)\otimes F_n(M).\]
	This shows that \(q_n([s(c_gg^{-1})gm])=[c_g]\otimes m\) and we are done.
\end{proof}
\begin{remark}
	In the other direction, assume that \(M\) is an \(\ff[G]\)-algebra in \(C\)-comodules and splittable.
	Then by Theorem A.9 we have \(h_r:M\simeq C\otimes P_1(M)\).
	Since \(1\in P_1(M)\) the map
	\[s_r:C\to M,\,c\mapsto h_r^{-1}(c\otimes 1)\]
	gives a map as in the above lemma.
\end{remark}

\section{Iwasawa (co)algebras and their (co)modules}\label{section:iwasawaCoalg}
Let \(\ff\) be a finite field of characteristic \(p\) and cardinality \(q=p^n\), and \(G\) a compact \(p\)-adic analytic group.

\begin{definition}
	The {\it Iwasawa algebra} \(\ff[\![G]\!]\) of \(G\) over \(\ff\) is defined to be the limit
	\[\ff[\![G]\!]:=\lim_{U}\ff[G/U]\]
	over all open normal subgroups \(U\) of \(G\).
	We equip it with the limit topology by giving the terms \(\ff[G/U]\) the discrete topology.

	The associated {\it Iwasawa coalgebra} is defined to be the colimit
	\[C^0(G,\ff):=\colim_U \mathrm{Set}(G/U,\ff)\]
	over all open normal subgroups \(U\) of \(G\).
\end{definition}
\begin{remark}
	As \(G\) is compact \(p\)-adic analytic, a subgroup of \(G\) is open iff it is of finite index in \(G\), see Theorem \ref{thm:openSubgroups}.
\end{remark}
\begin{lemma}
	An ideal of \(\ff[\![G]\!]\) is open iff it is of finite codimension.
\end{lemma}
\begin{proof}
	First, assume that \(I\) is open.
	By definition that means that \(I\) contains the kernel of \(\ff[\![G]\!]\to \ff[G/U]\) for some open normal subgroup \(U\subset G\).
	As \(G/U\) is finite this augmentation ideal is already of finite codimension, so \(I\) must be of finite codimension, too.

	In the other direction, assume that \(I\) has finite codimension, and denote the quotient algebra by \(A\).
	Since \(\ff\) is finite and \(A\) is finite dimensional over \(\ff\), the group of units \(A^\times\) is finite.
	Now the kernel \(U\) of the induced \(G\to A^\times\) is normal of finite index in \(G\), and thus is open (since \(G\) is \(p\)-adic analytic).
	Thus, \(I\) contains the kernel of \(\ff[\![G]\!]\to \ff[G/U]\), which is an open ideal, thus \(I\) must be open itself.
\end{proof}

\begin{lemma}
	There are natural isomorphisms
	\[C^0(G,\ff)^\vee\simeq \ff[\![G]\!]\text{ and }\ff[\![G]\!]^\circ\simeq C^0(G,\ff)\]
	of (co)algebras, where \(\vee\) denotes the \(\ff\)-linear dual and \(\circ\) the {\it restricted} dual. 
\end{lemma}
\begin{proof}
	For the first statement we calculate
	\[C^0(G,\ff)^\vee=\nb{\colim_UC^0(G/U,\ff)}^\vee=\lim_U C^0(G/U,\ff)^\vee=\lim_U \ff[G/U]=\ff[\![G]\!].\]
	For the second statement, recall that for an \(\ff\)-algebra \(R\) the restricted dual is defined to be the subspace \(R^\circ\subset R^\vee\) on those linear functionals whose kernel contains an ideal of finite codimension.
	Since those are exactly the open ideals, we can identify the restricted dual with the continuous dual, which is \(C^0(G,\ff)\).
\end{proof}
\begin{lemma}
	Let \(I\subset \ff[\![G]\!]\) be an ideal of finite codimension.
	Then the annihilator \(I^\perp\subset C^0(G,\ff)\) can be identified with \(\nb{\ff[\![G]\!]/I}^\vee\).
\end{lemma}
\begin{proof}
	An element \(f\in C^0(G,\ff)\) lies in \(I^\perp\) iff evaluation at \(f\) is zero for all elements in \(I\).
	Since \(I\) is of finite codimension, the inclusion of \(\nb{\ff[\![G]\!]/I}^\vee\) into the double dual of \(C^0(G,\ff)\) factors over \(C^0(G,\ff)\).
	Since their equality is evident after including both into the double dual, we are done.
\end{proof}
\begin{lemma}
	Let \(G\) have the form \(S\rtimes\ff^\times\) with \(S\) an open normal pro-\(p\) group.
	Then the Jacobson radical of \(\ff[\![G]\!]\) is the augmentation ideal
	\[J=I^G_S=\ker \ff[\![G]\!]\to \ff[\ff^\times]\]
	and \(\ff[\![G]\!]\) is complete with respect to \(J\).
\end{lemma}
\begin{proof}
	From \cite[prop.~5.2.16(iii)]{NeukirchSchmidtWingberg_CohomologyNumberFields} we know that \(\ff\pb{S}\) is local with maximal ideal the augmentation ideal \(I^S_S\).
	By \cite[thm.~4.2]{Passmann_InfiniteCrossed} we then get that the Jacobson radical \(J\) of \(\ff\pb{G}\) is \(I^S_S\star \ff^\times=I^G_S\).

	Since the coradical filtration on \(C^0(G,\ff)\) is exhaustive, we have that
	\[\ff\pb{G}=C^0(G,\ff)^\vee=\lim_n \nb{F_n(C^0(G,\ff))}^\vee=\lim_n C^0(G,\ff)^\vee/J^{n+1},\]
	where we used the identification \(F_n(C)^\perp=J^{n+1}\) from Construction \ref{cons:dualAlgebra}.
	Thus, \(\ff\pb{G}\) is complete with respect to \(J\)
\end{proof}

\begin{lemma}
	Let \(G\) have the form \(S\rtimes\ff^\times\) with \(S\) a pro-\(p\) group.
	Then the Iwasawa coalgebra \(C^0(G,\ff)\) is pointed, and the subcoalgebra \(F_0(C^0(G,\ff))\) is given by \(C^0(\ff^\times,\ff)\).
\end{lemma}
\begin{proof}
	Let \(C\subset C^0(G,\ff)\) be a simple subcoalgebra.
	Dualizing this we get a simple quotient \(\ff[\![G]\!]\to C^\vee\), meaning that the kernel is a maximal ideal.
	It thus contains the Jacobson radical, so we see that \(C^\vee\) is also a simple quotient of \(\ff[\![G]\!]/J=\ff[\ff^\times]\).
	Since \(\ff^\times\simeq C_{q-1}\) and the polynomial \(x^{q-1}-1\) splits into linear factors over \(\ff\), we see that we must have an isomorphism \(\ff\simeq C^\vee\).
	Thus, \(C\) is one dimensional over \(\ff\), showing that \(C^0(G,\ff)\) is pointed.

	The group-like elements are exactly the group homomorphisms \(G\to\ff^\times\).
	Since \(S\) is a pro-\(p\) group and \(\ff^\times\simeq C_{q-1}\) every such homomorphism must factor over \(G/S\simeq \ff^\times\).
\end{proof}

The categories of comodules of the Iwasawa coalgebra and modules of the Iwasawa algebra are closely linked:
\begin{proposition}\label{prop:iwasawaModuleComodule}
	The category of (left) \(C^0(G,\ff)\)-comodules is equivalent to the category of discrete (left) \(\ff[\![G]\!]\)-modules.
\end{proposition}
\begin{proof}
	Given a comodule \(M\) we define an action of \(G\) on \(M\) by
	\[G\times M\to M,\,(g,m)\mapsto (ev_{g^{-1}}\otimes \id)(\psi(m)).\]
	That this defines a discrete action follows from writing
	\[\psi(m)=\sum_{k=1}^{N}f_k\otimes m_k\]
	and observing that, since each \(f_k\) is continuous, they must all evaluate to \(1\) on some open normal subgroup \(U_k\).
	Thus, the stabilizer of \(m\) must contain the intersection of the \(U_k\) and thus must be open itself.

	In the other direction, assume that \(M\) is a discrete module and define a coaction by
	\[M\mapsto C^0(G,M)\simeq C^0(G,\ff)\otimes M,\,m\mapsto (g\mapsto g^{-1}.m).\]

	It is evident that these define functors and implement an equivalence between the categories.
\end{proof}
\begin{lemma}\label{lem:identifyP1}
	Under the equivalence of Proposition B.8, we have an equality
	\[P_1(M)=M^G\]
	of the primitives of \(M\) as a \(C^0(G,\ff)\)-comodule and the invariants of \(M\) as an \(\ff[\![G]\!]\)-module.
\end{lemma}
\begin{proof}
	This follows directly from Definition A.5 of the primitive elements of \(M\) and the formula for the action of \(G\) on \(M\) from Proposition B.8.
\end{proof}

\section{Continuous duality in Morava \(E\)-theory}\label{section:contdual}
Throughout, let \(E\) be a Morava \(E\)-theory associated to a formal group all of whose endomorphisms are already defined over the base field, and let \(K=E/\mm\) be the associated Morava \(K\)-theory.
We make this assumption since then all of the machinery developed in \cite{666} applies rather straightforwardly.
\begin{lemma}\label{lem:contdual1}
  Let \(X\) be a spectrum such that \(K_*X\) is concentrated in even degrees.
  Then the evaluation maps induce an equivalence
  \[E^0X\xrightarrow{\sim} Hom_{E_0}^{\mathrm{cont.}}\nb{E_0^\vee X, E_0}=Hom_{E_0}\nb{E_0^\vee X, E_0},\,f\mapsto (x\mapsto\langle x,f\rangle)\]
  where \(E_0^\vee X\) and \(E_0\) are given the 'natural topology' introduced in Section 11 of \cite{666}.
\end{lemma}
\begin{proof}
  Since \(E_*^\vee X\) is even and pro-free we get that the evaluations assemble to a bijection
  \[E^0X\xrightarrow{\sim}Hom_{E_0}\nb{E_0^\vee X,E_0}.\]
  It follows from 11.1-11.4 in \cite{666} that the above map factors through the inclusion of the continuous linear maps, which also shows that all \(E_0\)-linear maps from \(E_0^\vee X\) to \(E_0\) are continuous.
\end{proof}
\begin{lemma}\label{lem:contdual2}
  Let X be a spectrum such that \(K_*X\) is concentrated in even degrees.
  Then the evaluation maps induce an equivalence
  \[E_0^\vee X\xrightarrow{\sim}Hom_{E_0}^{\mathrm{cont.}}\nb{E^0X,E_0},\,x\mapsto (f\mapsto \langle x, f\rangle).\]
\end{lemma}
\begin{proof}
  Let \(\{x_i\}_{i\in I}\) be a pro-basis of \(E_0^\vee X\).
  By \ref{lem:contdual1} the evaluations at \(x_i\) give a bijection
  \[E^0X\xrightarrow{\sim}\prod_{i\in I}E_0.\]
  This is a continuous bijection between compact Hausdorff spaces by 11.1-11.5 in \cite{666}, and thus a homeomorphism.
  Under this homeomorphism evaluation at \(x_i\) goes to projection to the \(i\)th component.

  Consider the map
  \[ev:Hom_{E_0}^{\mathrm{cont.}}\nb{\prod_{i\in I}E_0,E_0}\to\prod_{i\in I}E_0,\,f\mapsto(i\mapsto f(\delta_i)).\]
  We claim that \(ev\) is injective: in fact, let \(ev(f)=0\) and \(F\) be the directed set of finite subsets of \(I\).
  For any \(c\in\prod_{i\in I}E_0\) consider the net
  \[N_C:A\to \prod_{i\in I}E_0,\, J\mapsto \sum_{j\in J}c_j\delta_j\]
  which converges to \(c\) in every projection, thus converges to \(c\).
  By linearity, we find that
  \[f(N_c(J))=\sum_{j\in J}c_jev(f)_j=0,\]
  by continuity of \(f\) we have \(f(c)=0\), and thus \(f=0\).

  We also claim that \(ev\) already lands in the submodule
  \[L_0\bigoplus_{i\in I}E_0\subset \prod_{i\in I}E_0,\]
  that is for all \(n\geq 0\) and for all but finitely many \(i\in I\) we have
  \[ev(f)_i\in\mm^n.\]
  Since the natural topology on \(E_0\) coincides with the \(\mm\)-adic one by \cite[prop.~11.9]{666} we find that \(f^{-1}(\mathfrak{m}^n)\) contains a neighborhood \(U\) of \(0\), showing that there are \(n_i\), with all but finitely many equal to \(0\), such that
  \[\prod_{i\in I}\mm^{n_i}\subset f^{-1}(\mm^n).\]
  Thus, for those \(i\) with \(n_i=0\) we have \(f(\delta_i)\in\mm^n\), which is exactly the claim.

  Now the composite
  \[E_0^\vee X\to Hom_{E_0}^{\mathrm{cont.}}\nb{\prod_{i\in I}E_0,E_0}\to L_0\bigoplus_{i\in I}E_0\]
  sends \(x_i\) to the \(i\)th pro-basis element, so it is an isomorphism.
  Since the second map is injective, both maps have to be bijections, completing the proof.
\end{proof}
\begin{remark}\label{rem:contdual2}
  The same argument also shows that, for any homeomorphism
  \[E^0X\xrightarrow{\sim}\prod_{j\in J}E_0,\]
  the elements \(x_j\in E_0^\vee X\) corresponding to the \(j\)th projection via the previous lemma form a pro-basis of \(E_0^\vee X\).
\end{remark}

\begin{lemma}\label{lem:manipulatingci}
  Let \(X\) be a spectrum such that \(K_*X\) is concentrated in even degrees.
  Then the map
  \[\pi_0 Map\nb{X,L_{K(2)}\bigoplus_{i\in I}E}\to \prod_{i\in I}E^0X\]
  is injective, and its image consists of those collections \(\{q_i\in E^0X\}_{i\in I}\) fulfilling the following finiteness condition: for all \(n\geq 0\) and \(x\in E^\vee_0 X\),  we have that
  \[\langle x, q_i\rangle \in \mm^n\]
  for all but finitely many \(i\in I\).
  %In particular, for any retraction \(r_M\) and any lifts \(C_i\in E^0MString\) of the associated \(c_i\in K^0MString\) satisfying the above finiteness condition, there is a unique (up to homotopy) \(K(2)\)-local \(E^{hF_{3/2}}\)-linear equivalence
  %\[MString\otimes E^{hF_{3/2}}\to \bigoplus_{i\in I}E.\]
\end{lemma}
\begin{proof}
  Using the universal coefficient theorem for \(K(2)\)-local \(E\)-modules \cite[thm.~4.1]{hoveyss} and since our assumption implies that \(E_*^\vee X\) is projective in \(L\)-complete \(E_*\)-modules, the above map can be identified with
  \[Hom_{E_0}\nb{E_0^\vee X, \nb{\bigoplus_{i\in I}E_0}_{\mm}^\wedge}\to Hom_{E_0}\nb{E_0^\vee X,\prod_{i\in I}E_0}.\]
  Thus, the injectivity statement and the identification of the image follow from the analogous statements for the map
  \[\nb{\bigoplus_{i\in I}E_0}_{\mm}^\wedge\to \prod_{i\in I}E_0.\]
  %So there is a unique (up to homotpy) map
  %\[MString\to L_{K(2)}\bigoplus_{i\in I}E\]
  %inducing the maps \(C_i\) after projections, and the unique \(E^{hF_{3/2}}\)-linearization is an equivalence by construction.
\end{proof}

\begin{lemma}\label{lem:algToSpectral}
	Let \(X\) be a spectrum such that \(K_*X\) is concentrated in even degrees, and let
	\[\alpha: K_0X\to \bigoplus_{i\in I}K_0E\]
	be an isomorphism of \(K_0E\)-comodules.
	Then there exists a map
	\[\beta: X\to L_K\bigoplus_{i\in I}E\]
	such that \(K_0\beta =\alpha\).
	In particular, \(\beta\) is a \(K\)-equivalence of spectra.
\end{lemma}
\begin{proof}
	Denote \(L_K\bigoplus_{i\in I} E\) by \(F\).
	By assumption \(E_0^\vee X\) is pro-free, so that we can choose a lift \(\tilde{\alpha}\) in \(E_0\)-modules such that
	% https://q.uiver.app/#q=WzAsNCxbMCwwLCJFXzBeXFx2ZWUgWCJdLFsyLDAsIkVfMF5cXHZlZSBGIl0sWzAsMSwiS18wWCJdLFsyLDEsIktfMEYiXSxbMiwzLCJcXGFscGhhIl0sWzAsMiwiIiwwLHsic3R5bGUiOnsiaGVhZCI6eyJuYW1lIjoiZXBpIn19fV0sWzEsMywiIiwyLHsic3R5bGUiOnsiaGVhZCI6eyJuYW1lIjoiZXBpIn19fV0sWzAsMSwiXFx0aWxkZXtcXGFscGhhfSIsMCx7InN0eWxlIjp7ImJvZHkiOnsibmFtZSI6ImRhc2hlZCJ9fX1dXQ==
	\[\begin{tikzcd}
		{E_0^\vee X} && {E_0^\vee F} \\
		{K_0X} && {K_0F}
		\arrow["{\tilde{\alpha}}", dashed, from=1-1, to=1-3]
		\arrow[two heads, from=1-1, to=2-1]
		\arrow[two heads, from=1-3, to=2-3]
		\arrow["\alpha", from=2-1, to=2-3]
	\end{tikzcd}\]
	commutes.
	Using the universal coefficient theorem for \(K\)-local \(E\)-modules \cite[thm.~4.1]{hoveyss}, and since both \(E_0^\vee F\) and \(E_0^\vee X\) are, by assumption, projective  in \(L\)-complete \(E_0\)-modules, we get a sequence of isomorphisms
	\[F^0F\simeq \hm_{E_0}(E_0^\vee F,F_0)\xrightarrow{\tilde{\alpha}^*}\hm_{E_0}(E_0^\vee X,F_0)\simeq F^0X.\]
	Let \(\beta\) be the image of \(\id_F\) under this.
	We now want to show that \(K_0\beta=\alpha\).
	For this it suffices to show that, for all \(i\in I\), we have \(K_0(pr_i\circ \beta)=pr_i\circ \alpha\).
	These are both maps \(K_0X\to K_0E\) of \(K_0E\)-comodules and, by construction, agree when composed with the counit \(K_0E\to K_0\).
	Dualizing this, they both send the unit of \(K^0E\) to the same element of \(K^0X\).
	Since both maps are \(K^0E\)-linear they must agree on all of \(K^0E\), and we are done.
\end{proof}
\begin{remark}
	Let \(q_i\in K^0X\) be the unique class such that
	\[\langle x, q_i\rangle=\epsilon(pr_i(\alpha(x)))\text{ for all }x\in K_0X.\]
	The above also shows that \(\beta_i=pr_i\circ \beta\in E^0X\) maps to \(q_i\) in \(K^0X\).
\end{remark}

We close with a lemma about mapping spaces into \(E_0\):
\begin{lemma}\label{lem:compactCompleteQuotients}
	Let \(X\) be a topological space, \(\ff\) a finite field of characteristic \(p\), and let
	\[R_d=W(\ff)\pb{x_1,\ldots,x_{d-1}}\]
	where \(d\geq1\).
	Let \(\mm_d= (p,x_1,\ldots,x_{d-1})\) be the maximal ideal and give \(R_d\) the \(\mm_d\)-adic topology.
	Then the natural map
	\[C^0(X,R_d)\to C^0(X, R_d/\mm_d^n)\]
	exhibits the target as \(C^0(X,R_d)/\mm_d^n\).
\end{lemma}
\begin{proof}
	To begin, note that \(R_d\) is a Noetherian regular local ring with finite quotient field, so each of the quotients \(R_d/\mm_d^n\) is finite.
	This shows that \(R_d\) is compact Hausdorff in the \(\mm_d\)-adic topology.
	We also have that the associated graded \(Gr_*R_d=\mm^*_d/\mm^{*+1}_d\) is a polynomial ring over \(\ff\) in \(d\) generators in degree \(1\).

	Let us now show surjectivity.
	As \(R_d/\mm_d^n\) is discrete any choice of section \(s\) of \(R_d\to R_d/\mm_d^n\) will be continuous, and the map
	\[C^0(X,R_d/\mm_d)\to C^0(X,R_d),\,f\mapsto s\circ f\]
	gives a section of \(C^0(X,R_d)\to C^0(X,R_d/\mm_d^n)\), so the latter must be surjective.

	We now identify the kernel.
	It always contains \(\mm_d^nC^0(X,R_d)\), and we need to show the other inclusion.
	The proof proceeds by induction in \(d\geq 1\).
	
	For the base case \(d=1\), let \(f\) be in the kernel of \(C^0(X,W(\ff))\to C^0(X,W(\ff)/p^n)\).
	This means that, for every \(x\in X\), \(p^n\) divides \(f(x)\), that is \(f\) factors over the closed subspace \(p^nW(\ff)\).
	Since \(W(\ff)\) is a domain the map
	\[W(\ff)\xrightarrow{\cdot p^n}p^nW(\ff)\]
	is a continuous bijection between compact Hausdorff spaces, thus a homeomorphism.
	Denote its inverse by \(T\).
	Then \(T\circ f\in C^0(X,W(\ff))\) and \(p^n(T\circ f)=f\), so that \(f\in p^nC^0(X,W(\ff))\).

	For the induction step from \(d\) to \(d+1\), consider the composite
	\[R_d\xrightarrow{i} R_{d+1}\xrightarrow{pr} R_{d+1}/(x_d),\]
	where \(i\) and \(pr\) are the standard inclusion and projection, which are continuous.
	The quotient topology on \(R_{d+1}/(x_d)\) coincides with the \(\mm_{d+1}/(x_d)\)-adic topology, and the composite is a continuous bijection between compact Hausdorff spaces, thus a homeomorphism.
	Denote its inverse by \(T\).

	Let \(f\) be in the kernel of \(C^0(X,R_{d+1})\to C^0(X,R_{d+1}/\mm_{d+1}^n)\).
	By the inductive assumptions, we have that \(T\circ pr\circ f\) is in \(m_d^nC^0(X,R_d)\), so \(f'=i\circ T\circ pr\circ f\) is in \((p,x_1,\ldots, x_{d-1})^nC^0(X,R_{d+1})\) and \(pr\circ f=pr\circ f'\).

	For every \(x\in X\), we have \((f-f')(x)\in (x_d)\).
	Since division by \(x_d\) is continuous (which can be shown as for \(p^n\) above) there is a unique \(f''\in C^0(X,R_{d+1})\) with \(x_df''=(f-f')\).
	Now let's do an induction in \(n\geq1\).

	For the base case \(n=1\) the above discussion shows that
	\[f=f'+x_d f''\in (p,x_1,\ldots,x_{d-1},x_d)C^0(X,R_{d+1}),\]
	and we are done.

	For the induction step from \(n-1\) to \(n\), note that  the above discussion implies \(x_df''(x)\in \mm_{d+1}^n\) for all \(x\in X\).
	Since the associated graded has no zero divisors this shows that already \(f''(x)\in \mm_{d+1}^{n-1}\) for all \(x\in X\), so that \(f''\in \mm_{d+1}^{n-1}C^0(X,R_{d+1})\) by the inductive hypothesis.
	We thus have
	\[f=f'+x_df''\in \nb{(p,x_1,\ldots,x_{d-1})^n+x_d\mm_{d+1}^{n-1}}C^0(X,R_{d+1})=\mm_{d+1}^nC^0(X,R_{d+1})\]
	and we are done.
\end{proof}

\addcontentsline{toc}{section}{References}
%\bibliographystyle{unsrt}
%\bibliography{bibliography}
\printbibliography
\end{document}